\documentclass[11pt]{article}
\usepackage{url}
\usepackage{amsfonts, mathrsfs}
\usepackage{amssymb}
\usepackage{amsmath}
\usepackage{amsthm}
\usepackage[lofdepth,lotdepth]{subfig}
\usepackage{setspace,amssymb,verbatim, enumitem}
\usepackage[margin=1in]{geometry}

\usepackage{lineno}

\usepackage{prettyref}
\newrefformat{cond}{condition (\ref{#1})}
\newrefformat{conj}{Conjecture \ref{#1}}
\usepackage{graphicx}
\usepackage{rotating}
\usepackage{psfrag}
\usepackage{afterpage}
\usepackage{multirow}
\usepackage[lofdepth,lotdepth]{subfig}
\usepackage{tikz}

\theoremstyle{plain}
\newtheorem{theorem}{Theorem}[section]
\newtheorem{proposition}[theorem]{Proposition}
\newtheorem*{theorem*}{Theorem}
\newtheorem*{theoremA}{Theorem A}

\newtheorem*{corollary*}{Corollary}
\newtheorem{corollary}[theorem]{Corollary}
\newtheorem{condition}[theorem]{Condition}
\newtheorem*{lemma*}{Lemma}
\newtheorem{lemma}[theorem]{Lemma}

\theoremstyle{remark}

\newtheorem*{case*}{Case}

\theoremstyle{definition}
\newtheorem{definition}{Definition}

\newtheorem{conjecture}{Conjecture}

\usepackage{prettyref}
\newrefformat{sec}{Section \ref{#1}}
\newrefformat{chap}{Chapter \ref{#1}}
\newrefformat{prop}{Proposition \ref{#1}}
\newrefformat{prob}{Problem \ref{#1}}
\newrefformat{rem}{Remark \ref{#1}}
\newrefformat{cor}{Corollary \ref{#1}}
\newrefformat{cond}{Condition \ref{#1}}
\newrefformat{def}{Definition \ref{#1}}
\newrefformat{ex}{Example \ref{#1}}
\newrefformat{lem}{Lemma \ref{#1}}
\newrefformat{eq}{Equation \eqref{#1}}
\newrefformat{con}{condition (\ref{#1})}
\newrefformat{conj}{Conjecture \ref{#1}}
\newrefformat{fig}{Figure \ref{#1}}
\newrefformat{tab}{Table \ref{#1}}
\newrefformat{app}{Appendix \ref{#1}}

\newcommand{\tw}[1]{{}^#1\!}
\newcommand{\legendre}[2]{\genfrac{(}{)}{}{}{#1}{#2}}

\newcommand{\Q}{\mathbb{Q}}
\newcommand{\K}{\mathbb{K}}

\newcommand{\aut}{\mathrm{Aut}}

\newcommand{\C}{\mathbb{C}}

\newcommand{\F}{\mathbb{F}}

\newcommand{\Hom}{\mathrm{Hom}}

\newcommand{\syl}{\mathrm{Syl}}

\newcommand{\gal}{\mathrm{Gal}}

\newcommand{\ind}{\mathrm{Ind}}
\newcommand{\irr}{\mathrm{Irr}}

\newcommand{\wh}[1]{\widehat{#1}}
\newcommand{\tr}{\mathrm{Tr}}
\newcommand{\Aut}{\mathrm{Aut}}

\newcommand{\End}{\mathrm{End}}

\newcommand{\la}{\lambda}
\newcommand{\fl}[1]{\mathfrak{F}(#1)}
\newcommand{\e}[1]{\End_{\F G}(\fl{#1})}
\newcommand{\wt}[1]{\widetilde{#1}}
\newcommand{\bG}[1]{\textbf{#1}}
\newcommand{\bg}[1]{\textbf{#1}}

\title{Galois Automorphisms on Harish-Chandra Series and Navarro's Self-Normalizing Sylow $2$-Subgroup Conjecture}
\author{A. A. Schaeffer Fry\\ \small\textit{Department of Mathematical and Computer Sciences}\\\small\textit{Metropolitan State University of Denver}\\ \small\textit{Denver, CO 80217, USA}\\ \small\textit{aschaef6@msudenver.edu}
}

\date{}
\begin{document}
\maketitle

\begin{abstract}
G. Navarro has conjectured a necessary and sufficient condition for a finite group $G$ to
have a self-normalizing Sylow 2-subgroup, which is given in terms of the ordinary irreducible
characters of $G$.   In a previous article, the author has reduced the proof of this conjecture to showing that
certain related statements hold for simple groups. In this article, we describe the action of Galois automorphisms on the Howlett--Lehrer parametrization of Harish-Chandra induced characters.  We use this to complete the proof of the conjecture by showing that the remaining simple groups satisfy the required conditions.  

\vspace{0.25cm}

\noindent \textit{Mathematics Classification Number:} 20C15, 20C33


\end{abstract}

\section{Introduction} 
Throughout, we write $\Q_n:=\Q(e^{2\pi i/n})$ to denote the extension field of $\Q$ obtained by adjoining $n$th roots of unity in $\C$.  In particular, if a finite group $G$ has size $n$, then $\Q_n$ is a splitting field for $G$ and the group $\mathrm{Gal}(\Q_n/\Q)$ acts on the set of irreducible ordinary characters, $\irr(G)$, of $G$.  

In \cite{navarro2004}, G. Navarro conjectured a refinement to the well-known McKay conjecture that incorporates this action of $\mathrm{Gal}(\Q_n/\Q)$.  Specifically, the ``Galois-McKay" conjecture posits that if $\ell$ is a prime, $\sigma\in\mathrm{Gal}(\Q_n/\Q)$ sends every $\ell'$ root of unity to some $\ell$th power, and $P\in\syl_\ell(G)$ is a Sylow $\ell$-subgroup of $G$, then the number of characters in $\irr_{\ell'}(G)$ that are fixed by $\sigma$ is the same as the number of characters in $\irr_{\ell'}(N_G(P))$ fixed by $\sigma$.  Here for a finite group $X$, we write $\irr_{\ell'}(X)=\{\chi\in\irr(X)\mid\ell\nmid\chi(1)\}$.

In the same paper, Navarro shows that the validity of his Galois-McKay conjecture would imply the following statement in the case $\ell=2$ (as well as a corresponding statement for $\ell$ odd):

\begin{conjecture}[Navarro]\label{conj:mainprob}
Let $G$ be a finite group and let $\sigma\in\gal(\Q_{|G|}/\Q)$ fixing $2$-roots of unity and squaring $2'$-roots of unity. Then $G$ has a self-normalizing Sylow $2$-subgroup if and only if every irreducible complex character of $G$ with odd degree is fixed by $\sigma$.
\end{conjecture}

The ordinary McKay conjecture has been reduced to proving certain inductive statements for simple
groups in \cite{IsaacsMalleNavarroMcKayreduction}, and even recently proven for $\ell=2$ by G. Malle and B. Sp{\"a}th in \cite{MalleSpathMcKay2}.  However, at the time of the writing of the current article, there is not yet such a reduction for the Galois-McKay conjecture.  In fact, very few groups have even been shown to satisfy the Galois-McKay conjecture, and a reduction seems even more elusive in the case $\ell=2$ than for odd primes.  We consider \prettyref{conj:mainprob} to be a weak form of the Galois-McKay refinement in this case, and hope that some of the  observations made in the course of its proof will be useful in future work on the full Galois-McKay conjecture.

In \cite{NavarroTiepTurull2007}, Navarro, Tiep, and Turull proved the corresponding statement for $\ell$ odd.  In \cite{SchaefferFrySN2S1}, the author proved a reduction theorem for \prettyref{conj:mainprob}, which reduced the problem to showing that slightly stronger inductive conditions hold for all finite nonabelian simple groups (groups satisfying these statements are called ``SN2S-Good"; see also \prettyref{sec:SimpleStatements} below for the statements), and proved that the sporadic, alternating, and several simple groups of Lie type satisfy these conditions.  In \cite{SFTaylorTypeA}, the author, together with J. Taylor, extended the strategy from \cite{SchaefferFrySN2S1} to show that the conditions hold for every simple group of Lie type in characteristic $2$, as well as for the simple groups $PSL_n(q)$ and $PSU_n(q)$ for odd $q$.  

The main result of this article is the completion of the proof of \prettyref{conj:mainprob}, which is achieved by showing that the remaining simple groups of Lie type satisfy the conditions to be SN2S-Good.  That is, we prove the following:

\begin{theoremA}\label{thm:mainthm}
\prettyref{conj:mainprob} holds for every finite group.
\end{theoremA}

Theorem A has the following interesting consequence:
\begin{corollary*}
One can determine from the character table of a finite group $G$ whether a Sylow $2$-subgroup of $G$ is self-normalizing.
\end{corollary*}

We remark that several relevant results regarding the Galois automorphism $\sigma$ involved in \prettyref{conj:mainprob}, including a simplified version of the reduction of \prettyref{conj:mainprob}, have been obtained by G. Navarro and C. Vallejo in \cite{navarro-vallejo:2017:2-local-blocks-with-one-simple-module}.  The latter appeared just before the submission of the current article, which is completely independent of the work in \cite{navarro-vallejo:2017:2-local-blocks-with-one-simple-module}.  In particular, the conditions to be proved in the case of simple groups of Lie type of type $B, C, D,$ and $\tw{2}{D}$, which are the primary focus here, remain nearly the same under Navarro and Vallejo's version of the reduction.  Specifically, our results here imply that simple groups of Lie type $B$, $C$, $D$, and $\tw{2}{D}$ satisfy \cite[Conjecture A]{navarro-vallejo:2017:2-local-blocks-with-one-simple-module}.  Further, the strategy employed here in these cases, and in particular our analysis of the action of the Galois group on the Howlett--Lehrer parametrization of characters of groups of Lie type (see \prettyref{sec:gallact} below) may be of significant independent interest, since the effect of group actions on such parametrizations is an especially problematic component in a number of main problems regarding the representations of groups of Lie type.

We begin in \prettyref{sec:gal twist} by studying the actions of field automorphisms on modules in a very general setting.  In \prettyref{sec:gallact}, we make use of this framework to describe the action of $\mathrm{Gal}(\Q_n/\Q)$ on the Harish-Chandra parameterization of characters of groups with a $BN$-pair, which builds upon the strategy employed by G. Malle and B. Sp{\"a}th in \cite{MalleSpathMcKay2} for the case of the action of $\aut(G)$.  (See \prettyref{thm:GaloisAct}.) We hope that this will be useful beyond the scope of this article, especially in the context of proving the Galois-McKay conjecture (or the inductive conditions in an eventual reduction to simple groups) for groups of Lie type.  In \prettyref{sec:SN2S}, we prove Theorem A by showing that the remaining simple groups of Lie type defined in odd characteristic are SN2S-Good.  This is done in \prettyref{cor:princseries1mod8}, and Theorems \ref{thm:typeC}, \ref{thm:typeBD}, and \ref{thm:typeE6}.

\section{Galois Twists }\label{sec:gal twist} 

Here we consider a general framework for studying the action of field automorphisms on modules.  Let $\F$ be a field. Throughout, an $\F$-algebra $A$ will be taken to mean a finite-dimensional associative unital $\F$-algebra, and an $A$-module will be taken to mean a finite-dimensional left $A$-module.  

\subsection{$\sigma$-Twists}

Let $\sigma\in\aut(\F)$ be an automorphism of $\F$.  For $\alpha\in\F$, we write $\alpha^\sigma$ for the image of $\alpha$ under $\sigma$.  Given a vector space $V$ over $\F$, we define the \emph{$\sigma$-twist}, $V^\sigma$, of $V$ to be the $\F$-vector space whose underlying abelian group is the same as $V$, but with scalar multiplication defined by
\[\alpha\star_\sigma v:=\alpha^{\sigma^{-1}} v\] for $\alpha\in\F$ and $v\in V$.  Here the product on the right-hand side is just the ordinary scalar multiplication on $V$.  For any $v\in V$, we denote by $v^\sigma$ the element $v$, but viewed in $V^\sigma$.  Hence, the scalar multiplication may alternatively be written:
\[\alpha v^\sigma:=\alpha^{\sigma^{-1}} v.\]

With this in place, we may analogously define the $\sigma$-twist of an $\F$-algebra $A$ to be the $\F$-algebra $A^\sigma$ whose underlying ring structure is the same as that of $A$, but with $\sigma$-twisted vector space structure defined as above.  For $a\in A$, we will denote by $a^\sigma$ the element $a$ but viewed in $A^\sigma$.  Further, for an $A$-module $M$, we may define the $\sigma$-twist of $M$ to be the $A^\sigma$-module $M^\sigma$ whose underlying vector space structure is the $\sigma$-twisted vector space structure as above, and for $a\in A, m\in M$, we have \[a^\sigma m^\sigma=(am)^\sigma.\]

Given two vector spaces $V$ and $W$ over $\F$, we define a \emph{$\sigma$-semilinear} map to be a homomorphism of abelian groups $\phi\colon V\rightarrow W$ satisfying $\phi(\alpha v)=\alpha^\sigma\phi(v)$ for all $\alpha\in\F$ and $v\in V$.

Note then that for a vector space $V$ over $\F$ or $\F$-algebra $A$ the maps 
\[-^\sigma\colon V\rightarrow V^\sigma \quad\hbox{ and }\quad -^\sigma\colon A\rightarrow A^\sigma\] defined by $v\mapsto v^\sigma$ and $a\mapsto a^\sigma$ are $\sigma$-semilinear isomorphisms of abelian groups and rings, respectively.  

\subsection{$\sigma$-Twists and Representations}

Let $\irr(A)$ denote the set of isomorphism classes of simple $A$-modules.   Then we see that the map $-^\sigma\colon \irr(A)\rightarrow \irr(A^\sigma)$, given by $M\mapsto M^\sigma$, is naturally a bijection.  

Let $M$ be an $A$-module and let $\rho\colon A\rightarrow \End_A(M)$ be the corresponding representation.  That is, for $a\in A$, $\rho(a)$ is the endomorphism of $M$ given by $\rho(a)m:=am$ for $m\in M$.  Fixing a basis $\mathfrak{B}$ for $M$, we write $[\rho(a)]_\mathfrak{B}\in \mathrm{Mat}_{\dim M}(\F)$ for the image of $a\in A$ under the resulting matrix representation.  The character of $A$ afforded by $M$ is the map $\eta\colon A\rightarrow\F$ defined by the trace $\eta(a):=\tr\left([\rho(a)]_\mathfrak{B}\right).$

Similarly, we have a representation $\rho^\sigma\colon A^\sigma\rightarrow \End_{A^\sigma}(M^\sigma)$ which affords the $A^\sigma$-module $M^\sigma$.  Here $\rho^\sigma(a^\sigma)m^\sigma:=a^\sigma m^\sigma=(am)^\sigma$ for $m\in M$, $a\in A$.  Also, note that if $\mathfrak{B}$ is a basis for $M$, then the set $\mathfrak{B}^\sigma\subseteq M^\sigma$ is a basis for $M^\sigma$.   If $\eta$ is the character of $A$ afforded by $M$, we will denote by $\eta^\sigma$ the character of $A^\sigma$ afforded by $M^\sigma$.

 Given a  matrix $X$ over $\F$, we denote by $X^\sigma$ the matrix obtained by applying $\sigma$ to each entry of $X$.  That is, if $X=(x_{ij})\in\mathrm{Mat}_d(\F)$ is a $d\times d$ matrix, the matrix $X^\sigma$ is $(x_{ij}^\sigma)\in\mathrm{Mat}_d(\F)$.  The following lemma tells us that the matrix corresponding to $\rho^\sigma(a^\sigma)$ is simply the matrix obtained in this way from $\rho(a)$.
 
 \begin{lemma}\label{lem:sigmamatrixrep}
 Let $M$ be an $A$-module affording the character $\eta$ of $A$ and let $\mathfrak{B}$ be a basis for $M$.    Given $a\in A$, we have
 \[\left[\rho^\sigma(a^\sigma)\right]_{\mathfrak{B}^\sigma}=\left[\rho(a)\right]_{\mathfrak{B}}^\sigma \in \mathrm{Mat}_{\dim M}(\F)\quad\hbox{ and }\quad \eta^\sigma(a^\sigma)=\eta(a)^\sigma.\]
 
 \end{lemma}
 \begin{proof}
 Let $a\in A$ and let $\mathfrak{b}\in\mathfrak{B}$ be an element of the basis.  Then we may write
 \[a\mathfrak{b}=\sum_{\mathfrak{c}\in\mathfrak{B}} \alpha_{\mathfrak{b}, \mathfrak{c}}\mathfrak{c}\] for some scalars $\alpha_{\mathfrak{b},\mathfrak{c}}\in\F$.  Then 
 \[a^\sigma\mathfrak{b}^\sigma=(a\mathfrak{b})^\sigma=\left(\sum_{\mathfrak{c}\in\mathfrak{B}} \alpha_{\mathfrak{b}, \mathfrak{c}}\mathfrak{c}\right)^\sigma=\sum_{\mathfrak{c}\in\mathfrak{B}} \alpha_{\mathfrak{b}, \mathfrak{c}}^\sigma\mathfrak{c}^\sigma,\] where the last equality follows from the $\sigma$-semilinearity of $-^\sigma\colon M\rightarrow M^\sigma$.  Hence $\left[\rho^\sigma(a^\sigma)\right]_{\mathfrak{B}^\sigma}=(\alpha_{\mathfrak{b},\mathfrak{c}}^\sigma)=\left[\rho(a)\right]_{\mathfrak{B}}^\sigma$, which proves the statement.
 \end{proof}
 
 \subsection{$\sigma$-Twists and Homomorphisms}
 
 Let $C\subseteq A$ be a subalgebra and let $M$ and $N$ be $A$-modules.  We denote by $\Hom_C(M,N)$ the vector space of all $\F$-linear maps $f\colon M\rightarrow N$ satisfying $f(cm)=cf(m)$ for every $c\in C$ and $m\in M$.  In particular, $\Hom_A(M,N)$ is the set of all $A$-module homomorphisms.  The following lemma and corollary, though straightforward, will be useful throughout the next several sections.
 
 \begin{lemma}\label{lem:sigmahoms}
Given $A$-modules $M$ and $N$ and a subalgebra $C$ of $A$, the map
\[\phi\colon\Hom_C(M, N)^\sigma\rightarrow\Hom_{C^\sigma}(M^\sigma, N^\sigma),\] defined by $\phi(f^\sigma)(m^\sigma):=f(m)^\sigma$ for all $f\in \Hom_C(M,N)$ and $m\in M$, is an isomorphism of vector spaces over $\F$.
 \end{lemma}
 
 \begin{corollary}\label{cor:sigmaend}
Given an $A$-module $M$, the map
\[\phi\colon\End_A(M)^\sigma\rightarrow\End_{A^\sigma}(M^\sigma),\] defined by $\phi(f^\sigma)(m^\sigma):=f(m)^\sigma$ for all $f\in \End_A(M)$ and $m\in M$, is an isomorphism of $\F$-algebras.
 \end{corollary}
 
 We remark that given a representation $\rho\colon A\rightarrow\End_A(M)$ and $a\in A$, the endomorphisms $\rho^\sigma(a^\sigma)$ and $\phi(\rho(a)^\sigma)$ agree, where $\phi$ is the isomorphism in \prettyref{cor:sigmaend}.
 
 \subsection{$\sigma$-Twists and Group Algebras}
 
 We now turn our attention to the case that $A=\F G$ is a group algebra over $\F$ for a finite group $G$.   Recall that members of $\F G$ are formal sums $\sum_{g\in G} \alpha_g g$, where $\alpha_g\in \F$.   
 
 
 We obtain a $\sigma^{-1}$-semilinear ring isomorphism $\iota_0\colon \F G\rightarrow \F G$ defined by 
 \[\iota_0\left(\sum_{g\in G} \alpha_g g\right):=\sum_{g\in G}\alpha_g^{\sigma^{-1}}g.\]  Composing with the $\sigma$-semilinear ring isomorphism $-^\sigma\colon \F G\rightarrow (\F G)^\sigma$ yields an isomorphism of $\F$-algebras $\iota_0^\sigma\colon \F G\rightarrow (\F G)^\sigma$ defined by $\iota_0^\sigma(a):=\iota_0(a)^\sigma$ for $a\in \F G$.
 
 If $M$ is an $\F G$-module, we may use this isomorphism to view the $(\F G)^\sigma$-module $M^\sigma$ as an $\F G$-module via  
 \[a\cdot m^\sigma:=\iota_0^\sigma(a)m^\sigma=\iota_0(a)^\sigma m^\sigma=(\iota_0(a)m)^\sigma\] for $a\in \F G$ and $m\in M$.
 
 By an abuse of notation, if $\rho\colon \F G\rightarrow \End_{\F G}(M)$ is the representation afforded by $M$, we will denote by $\rho^\sigma$ the representation of $\F G$ afforded by $M^\sigma$.  We remark that this can be viewed as $\iota_0^\sigma$ composed with the $(\F G)^\sigma$-representation $\rho^\sigma$.  
 
 Let $\F_0\subseteq\F$ denote the prime subfield of $\F$, and note that $\alpha^\sigma=\alpha$ for every $\alpha\in\F_0$.   In this context, we may rephrase \prettyref{lem:sigmamatrixrep} as follows:
 
 \begin{lemma}\label{lem:sigmamatrixrep2}
 Let $M$ be an $\F G$-module affording the character $\chi$ of $\F G$ and let $\mathfrak{B}$ be a basis for $M$.    Given $a\in \F G$, we have
 \[\left[\rho^\sigma(a)\right]_{\mathfrak{B}^\sigma}=\left[\rho(\iota_0(a))\right]_{\mathfrak{B}}^\sigma \in \mathrm{Mat}_{\dim M}(\F)\quad\hbox{ and }\quad \chi^\sigma(a)=\chi(\iota_0(a))^\sigma.\]  In particular, $\chi^\sigma(a)=\chi(a)^\sigma$ for $a\in\F_0 G$.

 \end{lemma}
 
 \subsection{$\sigma$-Twists on (Co)-induced Modules}\label{sec:twistcoind}
 
 Let $P\leq G$ be a subgroup, so that $\F P\subseteq \F G$ is a subalgebra.  Note that $\iota_0^\sigma$ restricts to an isomorphism $\iota_0^\sigma\colon \F P\rightarrow(\F P)^\sigma$.  Given an $\F P$-module $N$, we obtain the coinduced module $\mathfrak{F}_P^G(N):=\Hom_{\F P}(\F G, N)$, which is an $\F G$-module with action
\[(a f)(x):=f(xa),\] for $x, a\in\F G$.  Since $G$ is finite, this module is naturally isomorphic to the induced module. Hence, $\mathfrak{F}_P^G(N)$ affords the induced character $\ind_P^G(\lambda)$, where $\lambda$ is the character of $\F P$ afforded by $N$. 

\begin{lemma}\label{lem:sigmacoinduced}
Let $P\leq G$ be a subgroup and let $N$ be an $\F P$-module.  Then the map
\[\varphi\colon\mathfrak{F}_P^G(N)^\sigma\rightarrow\mathfrak{F}_P^G(N^\sigma),\] defined by $\varphi(f^\sigma)(a):=f(\iota_0(a))^\sigma$ for all $a\in \F G$ and $f\in\mathfrak{F}_P^G(N)$, is an isomorphism of $\F G$-modules.
\end{lemma}

\begin{proof}
From \prettyref{lem:sigmahoms}, we have an isomorphism of vector spaces
\[\phi\colon \mathfrak{F}_P^G(N)^\sigma=\Hom_{\F P}(\F G, N)^\sigma \rightarrow \Hom_{(\F P)^\sigma}((\F G)^\sigma, N^\sigma)= \Hom_{\F P} ((\F G)^\sigma, N^\sigma)\] such that for $x\in \F G$ and $f\in \mathfrak{F}_P^G(N)^\sigma$, $\phi(f^\sigma)(x^\sigma)=f(x)^\sigma$.  Then for $a\in \F G$, we have
\[\phi(af^\sigma)(x^\sigma)=\phi((\iota_0(a)f)^\sigma)(x^\sigma)=(\iota_0(a)f(x))^\sigma=f(x\iota_0(a))^\sigma\] and 
\[(a\phi(f^\sigma))(x^\sigma)=\phi(f^\sigma)(x^\sigma\iota_0^\sigma(a))=\phi(f^\sigma)((x\iota_0(a))^\sigma)=f(x\iota_0(a))^\sigma,\] so that $\phi$ is an isomorphism of $\F G$-modules.

Since $\iota_0^\sigma \colon\F G\rightarrow (\F G)^\sigma$ is an isomorphism of $\F G$-modules, we have an isomorphism of $\F G$-modules $\wt{\iota_0}^\sigma\colon \End_{\F P}((\F G)^\sigma, N^\sigma)\rightarrow \End_{\F P}(\F G, N^\sigma)$ defined by $\wt{\iota_0}^\sigma(f)(x):=f(\iota_0^\sigma(x))$.  The composition $\wt{\iota_0}^\sigma\circ\phi$ gives the desired isomorphism.
\end{proof}

Considering the endomorphism algebras, we obtain the following corollary:

\begin{corollary}\label{cor:sigmacoinducedend}
Let $P\leq G$ be a subgroup and let $N$ be an $\F P$-module.  There is an isomorphism of $\F$-algebras
\[\iota\colon\End_{\F G}(\mathfrak{F}_P^G(N))^\sigma\rightarrow \End_{\F G}(\mathfrak{F}_P^G(N^\sigma))\] given by $\iota(B^\sigma)(\varphi(f^\sigma))=\varphi((Bf)^\sigma)$ for $B\in\End_{\F G}(\mathfrak{F}_P^G(N))$ and $f\in \mathfrak{F}_P^G(N)$.
\end{corollary}
\begin{proof}
From \prettyref{cor:sigmaend} and \prettyref{lem:sigmacoinduced}, we have isomorphisms of $\F$-algebras $\End_{\F G}(\mathfrak{F}_P^G(N))^\sigma\cong \End_{(\F G)^\sigma}((\mathfrak{F}_P^G(N))^\sigma)\cong\End_{(\F G)^\sigma}(\mathfrak{F}_P^G(N^\sigma))=\End_{\F G}(\mathfrak{F}_P^G(N^\sigma))$.  Considering the definitions of the maps there, we obtain the result.
\end{proof}

\subsection{$\sigma$-Twists and Endomorphism Algebras}
Keeping the notation of the previous section, we let $\mathcal{H}_G(P,N)$ denote the opposite algebra \[\mathcal{H}_G(P,N):=\End_{\F G}(\mathfrak{F}_P^G(N))^\mathrm{opp}.\]  

Assume now that $\mathfrak{F}_P^G(N)$ is a semisimple $\F G$-module, so that $\mathcal{H}_G(P,N)$ is a semisimple $\F$-algebra.  Then we denote by $\irr(\F G|N)\subseteq\irr(\F G)$ the isomorphism classes of simple submodules of $\mathfrak{F}_P^G(N)$.  The map $\mathscr{H}_N\colon\irr(\F G|N)\rightarrow\irr(\mathcal{H}_G(P,N))$ given by \[\mathscr{H}_N(S):=\Hom_{\F G}(\mathfrak{F}_P^G(N),S)\] defines a bijection between the constituents of $\mathfrak{F}_P^G(N)$ and the irreducible $\mathcal{H}_G(P,N)$-modules.  The $\mathcal{H}_G(P,N)$-module structure for $\mathscr{H}_N(S)$ is given by 
\[Bf:=f\circ B\] for $B\in\mathcal{H}_G(P, N)$ and $f\in\mathscr{H}_N(S)$.

If $S$ is a simple $\F G$-submodule of $\mathfrak{F}_P^G(N)$, then $S^\sigma$ is a simple $\F G$-submodule of $\mathfrak{F}_P^G(N^\sigma)$, where we have made the identification $\mathfrak{F}_P^G(N)^\sigma\cong \mathfrak{F}_P^G(N^\sigma)$ from \prettyref{lem:sigmacoinduced}.  In this way, we may view $\mathscr{H}_{N^\sigma}(S^\sigma)$ as an $\mathcal{H}_G(P,N^\sigma)$-module.

Further, through the isomorphism $\mathcal{H}_G(P,N)^\sigma\cong \mathcal{H}_G(P, N^\sigma)$ guaranteed by \prettyref{cor:sigmacoinducedend}, we see that $\mathscr{H}_N(S)^\sigma$ may also be viewed as an $\mathcal{H}_G(P, N^\sigma)$-module.
 
 \begin{proposition}\label{prop:sigmasimplebij}
 Keeping the notation above, let $S$ be a simple $\F G$-submodule of $\mathfrak{F}_P^G(N)$.  Then there is an isomorphism 
 \[\phi\colon\mathscr{H}_N(S)^\sigma\rightarrow\mathscr{H}_{N^\sigma}(S^\sigma)\] of $\mathcal{H}_G(P,N^\sigma)$-modules.
 \end{proposition}
 
 \begin{proof}
 From \prettyref{lem:sigmahoms}, there is an isomorphism $\phi\colon\mathscr{H}_N(S)^\sigma\rightarrow\mathscr{H}_{N^\sigma}(S^\sigma)$ as vector spaces over $\F$ given by $\phi(f^\sigma)(m^\sigma):=f(m)^\sigma$ for $f\in \mathscr{H}_N(S)$ and $m\in\mathfrak{F}_P^G(N)$.  Note that, again, we have identified $\mathfrak{F}_P^G(N)^\sigma$ with $\mathfrak{F}_P^G(N^\sigma)$ through the isomorphism in \prettyref{lem:sigmacoinduced}.  Let $B\in\mathcal{H}_G(P, N)$, $f\in \mathscr{H}_N(S)$, and $m\in\mathfrak{F}_P^G(N)$.  Suppressing the notation for the isomorphisms $\varphi$ and $\iota$ in \prettyref{lem:sigmacoinduced} and \prettyref{cor:sigmacoinducedend} and simply identifying the relevant modules and algebras, we may write $B^\sigma (m^\sigma)=(B(m))^\sigma$.
 
 Now, we have 
 \[\phi(B^\sigma f^\sigma)(m^\sigma) = \phi(f^\sigma\circ B^\sigma)(m^\sigma)=\phi((f\circ B)^\sigma)(m^\sigma)=((f\circ B)(m))^\sigma = f(B(m))^\sigma,\] but also
 \[B^\sigma\phi(f^\sigma)(m^\sigma) = (\phi(f^\sigma)\circ B^\sigma)(m^\sigma)=\phi(f^\sigma)(B^\sigma(m^\sigma))=\phi(f^\sigma)(B(m)^\sigma)=f(B(m))^\sigma,\] so that $\phi$ is an isomorphism of $\mathcal{H}_G(P, N^\sigma)$-modules, as stated.
 \end{proof}

\section{The Action of Galois Automorphisms on Harish-Chandra Series}\label{sec:gallact}

In \cite{MalleSpathMcKay2}, G. Malle and B. Sp{\"a}th analyze the action of $\aut(G)$ on Harish-Chandra induced characters.  In this section, our aim is to apply the framework introduced in \prettyref{sec:gal twist} to provide analogous statements in the case of $\mathrm{Gal}(\Q_{|G|}/\Q)$ acting on characters lying in a given Harish-Chandra series.  Throughout, we primarily follow the notation and treatment found in \cite{MalleSpathMcKay2} and \cite[Chapter 10]{carter2}.

For finite groups $Y\leq X$ and $\psi\in\irr(Y)$  we write $\irr(X|\psi)$ for the constituents of the induced character $\ind_Y^X(\psi)$ and $X_\psi$ for the stabilizer in $X$ of $\psi$.  Further, for $\chi\in\irr(X)$, we write $\irr(Y|\chi)$ for the constituents of the restricted character $\mathrm{Res}_Y^X(\chi):=\chi|_Y$.  

Let $G$ be a finite group of Lie type, realized as the group of fixed points $\bG{G}^F$ of a connected reductive algebraic group $\bf{G}$ under a Frobenius endomorphism $F$.  Let $\bf{P}$ be an $F$-stable parabolic subgroup of $\bf{G}$ with Levi decomposition $\bf{P}=\bf{L}\bf{U}$, where $\bf{L}$ is an $F$-stable Levi subgroup of $\bf{G}$.  Write $L=\bG{L}^F$ and $P=\bG{P}^F$ for the corresponding Levi and parabolic subgroups, respectively, of $G$.  For $\lambda\in\irr(L)$ a cuspidal character, we obtain the Harish-Chandra induced character $R_L^G(\lambda)$ of $G$ via inflation of $\lambda$ to $P$ followed by inducing to $G$, and the irreducible constituents $\irr\left(G|R_L^G(\lambda)\right)$ of this character make up the $(L,\lambda)$ Harish-Chandra series of $G$.  In particular, when $\bG{P}=\bG{B}$ is an $F$-stable Borel subgroup and $L=T$ is a maximally split torus, the irreducible constituents of $R_T^G(\lambda)$ make up the so-called principal series of $G$, which we will  be particularly interested in later due to results of \cite{MalleSpathMcKay2}.  

It is known from the work of Howlett and Lehrer (see \cite{howlettlehrer80, howlettlehrer83}) that the members of $\irr\left(G|R_L^G(\lambda)\right)$ are in bijection with characters $\eta\in\irr(W(\lambda))$, where $W(\lambda)$ is the so-called {relative Weyl group with respect to $\lambda$}.  (We note that in fact, the original bijection required a ``twist" by a certain $2$-cycle, but that it was shown later that one could take this $2$-cycle to be trivial - see \cite{geck93}.)  Write $R_L^G(\lambda)_\eta$ for the constituent of $R_L^G(\lambda)$ corresponding to $\eta\in\irr(W(\lambda))$ through this bijection.  The goal of this section is to understand how $\sigma\in\gal(\Q_{|G|}/\Q)$ acts on these characters, in terms of this labeling.  In what follows, we will utilize the construction in \cite[Chapter 10]{carter2} for the appropriate modules and bijections.

\subsection{The General Setting}

Here we introduce the notation and setting that will be used throughout the remainder of \prettyref{sec:gallact}.  In particular, we adapt the more general setting as in \cite[Section 4]{MalleSpathMcKay2} of a finite group $G$ with a split $BN$-pair of characteristic $p$.  As in there, we assume that the Weyl group $W=N/(N\cap B)$ is of crystallographic type; $\Phi$ and $\Delta$ denote the root system and simple roots of $W$, respectively; and for $\alpha\in\Phi$, $s_{\alpha}\in W$ is the corresponding reflection.  

We fix a standard parabolic subgroup $P=UL$ of $G$ with Levi subgroup $L$ and unipotent radical $U$ and denote by $N(L)$ the group $N(L)=(N_G(L)\cap N)L$.  Further, $W(L)=N(L)/L$ is the relative Weyl group of $L$ in $G$.  We also fix a cuspidal character $\lambda\in\irr(L)$, afforded by a representation $\rho$ for $L$, and write $W(\lambda):=N(L)_\lambda/L$.  By an abuse of notation, we also denote by $\lambda$ and $\rho$ the irreducible character and representation, respectively, of $P$ obtained by inflation.   

Let $\F$ be a splitting field for $G$ which is a finite Galois extension of $\Q$ containing $\sqrt{p}$, and let $\sigma\in\aut(\F)$ be any automorphism of $\F$.  Let $M$ be an irreducible left $\F P$-module affording $\lambda$ and let $\rho$ denote the corresponding representation over $\F$.  We can construct a module affording $R_L^G(\lambda)=\ind_P^G(\la)$ by taking $\mathfrak{F}(\lambda):=\Hom_{\F P}(\F G, M)$ as in \prettyref{sec:twistcoind}. Then $R_L^G(\la^\sigma)$ is afforded by $\mathfrak{F}(\lambda^\sigma)=\Hom_{\F P}(\F G, M^\sigma)$.

Finally, as in  \cite[Section 4]{MalleSpathMcKay2}, $\Lambda$ will represent a fixed $N(L)$-equivariant extension map for $L\lhd N(L)$, which exists by \cite{geck93} and \cite[Theorem 8.6]{Lusztig84}.  In particular, this means that for each $\lambda\in\irr(L)$, $\Lambda(\lambda)=:\Lambda_{\lambda}$ is an irreducible character of $N(L)_\lambda$ that extends $\lambda$.

We remark that by definition of the action of $\sigma$ on $\lambda$, we have $N(L)_\lambda=N(L)_{\lambda^\sigma}$.  For each $w\in W(\lambda):=N(L)_{\lambda}/L$, we fix a preimage $\dot{w}\in N(L)$.  
Now, notice that $(\Lambda_\lambda)^\sigma$ and $\Lambda_{\lambda^\sigma}$ are two extensions of $\la^\sigma$ to $N(L)_{\lambda}$ that do not necessarily agree on $N(L)_{\lambda}\setminus L$, so that by Gallagher's theorem (see, for example, \cite[Corollary 6.17]{isaacs}), there is some well-defined linear character $\delta_{\lambda, \sigma}\in\irr(W(\la))$ such that 

\[(\Lambda_\lambda)^\sigma=\delta_{\lambda, \sigma}\Lambda_{\la^\sigma}.\]

Let $\widetilde{\rho}$ be an extension of $\rho$ affording $\Lambda_\la$.  Note that $\widetilde{\rho}':=\delta_{\la,\sigma}^{-1}\widetilde{\rho}^\sigma$ is then an extension of $\rho^\sigma$ affording $\Lambda_{\la^\sigma}$.

\subsection{Action on Basis Elements }

As in \cite[10.1.4-5]{carter2}, there is a basis for $\End_{\F G}(\fl{\la})$ comprised of elements $B_{w,\lambda}$ for $w\in W(\la)$, where $B_{w,\lambda}\in\e{\la}$ is defined by 

\[(B_{w,\lambda}f)(x):=\widetilde{\rho}(\dot{w}) f(\dot{w}^{-1} e_U x)\] for all $f\in \fl{\lambda}$ and $x\in \F G$, where \[e_U:=\frac{1}{|U|}\sum_{u\in U} u\] and $U$ is the unipotent radical of $P$.

Similarly, $\e{\la^\sigma}$ has a basis comprised of $B_{w,\lambda^\sigma}$ with \[(B_{w,\lambda^\sigma}f)(x):=\widetilde{\rho}'(\dot{w}) f(\dot{w}^{-1} e_U x)\] for all $f\in \fl{\lambda^\sigma}$ and $x\in \F G$.

Recall from \prettyref{sec:gal twist} that we have a $\sigma$-semilinear isomorphism of rings \[-^\sigma\colon\e{\la}\rightarrow\e{\la^\sigma},\] where we identify $\fl{\la}^\sigma$ with $\fl{\la^\sigma}$ and $\e{\la}^\sigma$ with $\e{\la^\sigma}$ according to \prettyref{lem:sigmacoinduced} and \prettyref{cor:sigmacoinducedend}. Throughout, we will continue to make these identifications and to suppress the notation of the isomorphisms $\varphi$ and $\iota$ introduced there. We observe in the next lemma how the basis elements $B_{w,\la}$ and $B_{w,\la^\sigma}$ are related under this map.

\begin{lemma}\label{lem:iotaBw}
For each $w\in W(\la)$, we have
\[B_{w,\la}^\sigma=\delta_{\la,\sigma}(w)B_{w,\la^\sigma}.\]
\end{lemma}
\begin{proof}
Let $f\in\fl{\la}$ and $a\in\F G$.  Using the descriptions above, \prettyref{cor:sigmacoinducedend}, and \prettyref{lem:sigmacoinduced}, we see \[B_{w,\la}^\sigma f^\sigma(a)=(B_{w,\la}f)^\sigma(a)=B_{w, \la} f(\iota_0(a))^\sigma=\left(\wt{\rho}(\dot{w})f(\dot{w}^{-1}e_U\iota_0(a))\right)^\sigma.\]  Further, since $\dot{w}^{-1}e_U\in\Q G$, and hence is fixed by $\iota_0$, this means
\[B_{w,\la}^\sigma f^\sigma(a)=\delta_{\lambda,\sigma}({w})\wt{\rho}'(\dot{w})f(\iota_0(\dot{w}^{-1}e_Ua))^\sigma=\delta_{\la,\sigma}({w})\wt{\rho}'(\dot{w})f^\sigma(\dot{w}^{-1}e_Ua)=\delta_{\la,\sigma}({w})(B_{w,\la^\sigma}f^\sigma)(a),\] which proves the statement.
\end{proof}

We define $\Omega, p_{\alpha,\lambda}, \Phi_{\lambda}, \Delta_{\lambda}, R(\lambda),$ and $C(\lambda)$ as in \cite[4.B]{MalleSpathMcKay2}, so that $W(\lambda)=C(\lambda)\ltimes R(\lambda)$ and $R(\lambda)=\langle s_{\alpha} \mid \alpha\in \Phi_\lambda\rangle$ is a Weyl group with simple system $\Delta_{\lambda}$.    

To be more precise, $\Omega$ is the set
\[\Omega=\{\alpha\in\Phi\setminus\Phi_L\mid w(\Delta_L\cup\{\alpha\})\subseteq\Delta\hbox{ for some $w\in W$ and }(w_0^Lw_0^{\alpha})^2=1\}.\]  Here $w_0^L, w_0^\alpha$ are the longest elements in $W(L)$ and $\langle W(L), s_\alpha\rangle$, respectively, and $\Phi_L\subseteq\Phi$ is the root system of $W(L)$ with simple system $\Delta_L\subseteq\Delta$.  Then for $\alpha\in\Omega$, letting $L_\alpha$ denote the standard Levi subgroup of $G$ with simple system $\Delta_L\cup\{\alpha\}$, $L$ is a standard Levi subgroup of $L_\alpha$ and  $p_{\alpha,\la}\geq1$ is defined to be the ratio between the degrees of the two constituents of $R_L^{L_\alpha}(\la)$.  Then $\Phi_{\la}$ is the subset of $\Omega$ consisting of $\alpha$ such that $s_\alpha\in W(\la)$ and $p_{\alpha,\la}\neq1$.  This is a root system with simple roots $\Delta_\la\subseteq\Phi_\la\cap\Phi^+$.  Then $R(\lambda)=\langle s_{\alpha} \mid \alpha\in \Phi_\lambda\rangle$ is a Weyl group with simple system $\Delta_{\lambda}$, $C(\lambda)$ is the stabilizer of $\Delta_\la$ in $W(\la)$, and $W(\lambda)=C(\lambda)\ltimes R(\lambda)$ .

We see immediately from the definitions that:

\begin{lemma}
For $\alpha\in\Omega$, $p_{\alpha,\lambda}=p_{\alpha,\lambda^\sigma}$.  Hence, we have $\Phi_{\lambda}=\Phi_{\la^\sigma}$ and $\Delta_{\la}=\Delta_{\la^\sigma}$.  Further, $R(\lambda^\sigma)=R(\lambda)$ and $C(\lambda^\sigma)=C(\lambda)$.
\end{lemma}

Now, for each $\alpha\in\Delta_{\lambda}$, let $\epsilon_{\alpha,\lambda}\in\{\pm1\}$ be as in \cite[(4.3)]{MalleSpathMcKay2}.  Namely, we have

\[B_{s_\alpha,\lambda}^2=\ind(s_\alpha)\cdot \mathrm{id}+\epsilon_{\alpha,\lambda}\frac{p_{\alpha,\lambda}-1}{\sqrt{\ind(s_{\alpha})p_{\alpha,\lambda}}}B_{s_{\alpha},\lambda},\] and similarly
\[B_{s_\alpha,\lambda^\sigma}^2=\ind(s_\alpha)\cdot \mathrm{id}+\epsilon_{\alpha,\lambda^\sigma}\frac{p_{\alpha,\lambda}-1}{\sqrt{\ind(s_{\alpha})p_{\alpha,\lambda}}}B_{s_{\alpha},\lambda^\sigma}.\]

Here, for any $w\in W$, we write $\ind(w):=|U_0\cap U_0^{w_0w}|$ where $U_0$ is the unipotent radical of the subgroup $B$ and $w_0$ is the longest element of $W$.  We remark that since $p_{\alpha,\la}$ and $\ind(w)$ are powers of $p$ (see \cite[Theorem 10.5.5]{carter2}), the square root $\sqrt{\ind(s_{\alpha})p_{\alpha,\lambda}}$ lies in $\F$ by our assumption $\sqrt{p}\in\F$.

 Using this, we may further relate $\epsilon_{\alpha,\lambda}$ to $\epsilon_{\alpha,\lambda^\sigma}$:

\begin{lemma}\label{lem:epsilons}

For each $\alpha\in\Delta_{\lambda}$,  \[\displaystyle{\frac{\epsilon_{\alpha,\lambda}}{\epsilon_{\alpha,\lambda^\sigma}}=\delta_{\la,\sigma}(s_\alpha)\frac{\left(\sqrt{\ind(s_\alpha)p_{\alpha,\lambda}}\right)^\sigma}{\sqrt{\ind(s_\alpha)p_{\alpha,\lambda}}}}.\]
\end{lemma}

\begin{proof}
Notice that from \prettyref{lem:iotaBw} and the $\sigma$-semilinearity of $-^\sigma\colon\e{\la}\rightarrow\e{\la^\sigma}$, we see
 \[(B_{s_{\alpha},\lambda}^2)^\sigma=\ind(s_{\alpha})\cdot\mathrm{id}+\epsilon_{\alpha,\lambda}\frac{p_{\alpha,\lambda}-1}{\left(\sqrt{\ind(s_{\alpha})p_{\alpha,\lambda}}\right)^\sigma}B_{s_\alpha,\lambda}^\sigma\]
 
 \[=\ind(s_{\alpha})\cdot\mathrm{id}+\epsilon_{\alpha,\lambda}\frac{p_{\alpha,\lambda}-1}{\left(\sqrt{\ind(s_{\alpha})p_{\alpha,\lambda}}\right)^\sigma}\delta_{\lambda,\sigma}(s_{\alpha})B_{s_\alpha,\lambda^\sigma}\] and this is the same as

\[(\delta_{\lambda,\sigma}(s_{\alpha}))^2B_{s_\alpha,\lambda^\sigma}^2=(\delta_{\lambda,\sigma}(s_{\alpha}))^2\ind(s_\alpha)\cdot \mathrm{id}+(\delta_{\lambda,\sigma}(s_{\alpha}))^2\epsilon_{\alpha,\lambda^\sigma}\frac{p_{\alpha,\lambda}-1}{\sqrt{\ind(s_{\alpha})p_{\alpha,\lambda}}}B_{s_{\alpha},\lambda^\sigma}.\]  
Since $s_{\alpha}^2=1$ and $\delta_{\la,\sigma}$ is linear, this means $\epsilon_{\alpha,\lambda}\frac{p_{\alpha,\lambda}-1}{\left(\sqrt{\ind(s_{\alpha})p_{\alpha,\lambda}}\right)^\sigma}\delta_{\la,\sigma}(s_\alpha)=\epsilon_{\alpha,\lambda^\sigma}\frac{p_{\alpha,\lambda}-1}{\sqrt{\ind(s_{\alpha})p_{\alpha,\lambda}}}$, completing the proof.
\end{proof}


We now wish to determine the image under $\sigma$ of the alternate basis $T_{w,\lambda}$ for $\e{\la}$.  For $\alpha\in\Delta_{\lambda}$, the element $T_{s_\alpha,\lambda}$ is defined by $T_{s_\alpha,\la}:=\epsilon_{\alpha,\lambda}\sqrt{\ind(s_\alpha)p_{\alpha,\lambda}}B_{s_\alpha,\lambda}$.  For $w\in C(\lambda)$, $T_{w,\lambda}:=\sqrt{\ind(w)}B_{w,\lambda}$, and for $w=w_1s_{\alpha_1}\cdots s_{\alpha_r}$ where $w_1\in C(\lambda)$ and $s_{\alpha_1}\cdots s_{\alpha_r}$ is a reduced expression in $R(\lambda)$, $T_{w,\lambda}$ is defined by
\begin{equation}\label{eq:Tw}
T_{w,\lambda}:=T_{w_1,\lambda}T_{s_{\alpha_1},\lambda}\cdots T_{s_{\alpha_r},\lambda}.\end{equation}  We remark that this definition is independent of choice of reduced expression.  (See \cite[Proposition 10.8.2]{carter2}.)

  This yields the following:

\begin{proposition}\label{prop:iotaTw}
Let $w\in W(\lambda)$ have the form $w=w_1w_2$ for $w_1\in C(\lambda)$ and $w_2\in R(\lambda)$.  Then 
\[T_{w,\lambda}^\sigma=\frac{\sqrt{\ind(w_1)}^\sigma}{\sqrt{\ind(w_1)}}\delta_{\lambda,\sigma}(w_1)T_{w,\lambda^\sigma}.\]
\end{proposition}

\begin{proof}

We may write $w_2\in R(\lambda)$ as a reduced expression $w_2=s_{\alpha_1}\cdots s_{\alpha_r}$ as above.  First note that by \prettyref{lem:iotaBw} and $\sigma$-semilinearity, 

\[T_{w_1,\lambda}^\sigma = (\sqrt{\ind(w_1)}B_{w_1,\lambda})^\sigma = \sqrt{\ind(w_1)}^\sigma B_{w_1,\lambda}^\sigma = \sqrt{\ind(w_1)}^\sigma\delta_{\lambda,\sigma}(w_1)B_{w_1,\lambda^\sigma},\] and for each $i$,

\[T_{s_{\alpha_i},\lambda}^\sigma= \left(\epsilon_{\alpha_i,\lambda}\sqrt{\ind(s_{\alpha_i})p_{\alpha_i,\lambda}}B_{s_{\alpha_i},\lambda}\right)^\sigma =\epsilon_{\alpha_i,\lambda}\left(\sqrt{\ind(s_{\alpha_i})p_{\alpha_i,\lambda}}\right)^\sigma  B_{s_{\alpha_i},\lambda}^\sigma \]\[=\epsilon_{\alpha_i,\lambda}\left(\sqrt{\ind(s_{\alpha_i})p_{\alpha_i,\lambda}}\right)^\sigma \delta_{\lambda,\sigma}(s_{\alpha_i})B_{s_{\alpha_i},\lambda^\sigma} . \]  
But notice $\delta_{\lambda,\sigma}(w_1)T_{w_1,\lambda^\sigma}=\delta_{\lambda,\sigma}(w_1)\sqrt{\ind(w_1)}B_{w_1,\lambda^\sigma},$ so that \[T_{w_1,\lambda}^\sigma =  \frac{\sqrt{\ind(w_1)}^\sigma}{\sqrt{\ind(w_1)}}\delta_{\lambda,\sigma}(w_1)T_{w_1,\lambda^\sigma}. \]
Moreover,
$\delta_{\lambda,\sigma}(s_{\alpha_i})T_{s_{\alpha_i},\lambda^\sigma}=\epsilon_{\alpha_i,\lambda^\sigma}\sqrt{\ind(s_{\alpha_i})p_{\alpha_i,\lambda}}\delta_{\lambda,\sigma}(s_{\alpha_i})B_{s_{\alpha_i},\lambda^\sigma},$ and hence
\[T_{s_{\alpha_i},\lambda}^\sigma=\frac{\epsilon_{\alpha_i,\lambda}\left(\sqrt{\ind(s_{\alpha_i})p_{\alpha_i,\lambda}}\right)^\sigma}{\epsilon_{\alpha_i,\lambda^\sigma}\sqrt{\ind(s_{\alpha_i})p_{\alpha_i,\lambda}}}\delta_{\lambda,\sigma}(s_{\alpha_i})T_{s_{\alpha_i},\lambda^\sigma}=T_{s_{\alpha_i},\lambda^\sigma},\] where the last equality is by \prettyref{lem:epsilons}. Finally, since $\sigma$ is an isomorphism of rings, the statement follows from \eqref{eq:Tw}. 
\end{proof}

For $w=w_1w_2\in W(\lambda)$ with $w_1\in C(\lambda)$ and $w_2\in R(\lambda)$, we define $r_\sigma(w):=\frac{\sqrt{\ind(w_1)}^\sigma}{\sqrt{\ind(w_1)}}$.  As may be inferred from \prettyref{prop:iotaTw}, the term $r_\sigma(w)$ will play an important role.   Notice that $r_\sigma(w)\in\{\pm1\}$.  Indeed, for $z\in\Q$, $\sqrt{z}^\sigma$ must be a solution to the polynomial $t^2-z=0$, and hence $\sqrt{z}^\sigma\in\{\pm\sqrt{z}\}$. 

In some cases, it is clear that $r_\sigma\colon W(\lambda)\rightarrow\{\pm1\}$ defines a homomorphism, and hence a linear character of $W(\lambda)$.  When this is the case, note that $r_\sigma(w)=r_\sigma(w)^{-1}=r_\sigma({w^{-1}})$ and that $r_\sigma$ is simply a character of $C(\la)\cong W(\la)/R(\la)$ inflated to $W(\la)$.  In particular, this holds in the case that $G$ is a group of Lie type, by \cite[Section 2.9]{carter2} (see \prettyref{lem:rw} below and the discussion preceding it).

We remark that in particular, if $R(\lambda)\leq \ker(\delta_{\lambda,\sigma})$, then 
$T_{w,\la}^\sigma=r_\sigma(w)\delta_{\la,\sigma}(w)T_{w,\la^\sigma}.$ 
However, more generally, we may inflate the linear character $\delta_{\la,\sigma}|_{C(\la)}$ of $C(\la)\cong W(\la)/R(\la)$ to another linear character $\delta_{\la,\sigma}'$ of $W(\la)$ that is trivial on $R(\la)$.   This gives the following corollary:

\begin{corollary}\label{cor:iotaTw}
Let $w\in W(\lambda)$.  Then 
\[T_{w,\lambda}^\sigma=r_\sigma(w)\delta'_{\lambda,\sigma}(w)T_{w,\lambda^\sigma}.\]
\end{corollary}

%
%
%

\subsection{Action on Modules}

We now turn our attention to $\e{\la}$-modules and prove an analogue to \cite[Proposition 4.5]{MalleSpathMcKay2}.   


\begin{proposition}\label{prop:moduleact}
Let $\eta$ be an irreducible character of $\e{\la}$.  Then the irreducible character $\eta^\sigma$ of $\e{\la^\sigma}$ satisfies  

\[\eta^\sigma(T_{w,\lambda^\sigma})={r}_{\sigma}(w)\delta'_{\lambda,\sigma}(w^{-1})\eta(T_{w,\la})^\sigma\]
for every $w\in W(\lambda)$.
\end{proposition}  
\begin{proof}
From \prettyref{lem:sigmamatrixrep} and \prettyref{cor:iotaTw}, we have
 \[\eta(T_{w,\la})^\sigma=\eta^\sigma(T_{w,\la}^\sigma)=\eta^\sigma(r_\sigma(w)\delta'_{\la,\sigma}(w)T_{w,\la^{\sigma}}).\]  Since $r_\sigma(w)\in\{\pm1\}$ and $\delta'_{\la,\sigma}$ is a linear character, this yields the statement.
 \end{proof}

\subsection{The Generic Algebra}

Here we discuss and set the notation for the generic algebras and specializations that yield the bijection $\mathfrak{f}\colon\irr(\e{\la})\rightarrow\irr(W(\la))$ in Howlett--Lehrer theory.  Like before, we follow largely the notation and discussion in \cite[4.D]{MalleSpathMcKay2}.  Let $\textbf{u}=(u_\alpha\colon\alpha\in\Delta_\lambda)$ be indeterminates such that $u_\alpha=u_\beta$ if and only if $\alpha$ is $W(\lambda)$-conjugate to $\beta$.  Let $A_0:=\C[\bf{u},\bf{u}^{-1}]$ and let $K$ be an algebraic closure of the quotient field of $A_0$ and $A$ the integral closure of $A_0$ in $K$.  We define the generic algebra $\mathcal{H}$ to be a free $A$-module with basis $\{a_w\colon w\in W(\lambda)\}$ equipped with the $A$-bilinear associative multiplication satisfying the properties as in \cite[5.D]{MalleSpathMcKay2}.  

We denote by $\mathcal{H}^K$ the $K$-algebra $K\otimes_A \mathcal{H}$, and for any ring homomorphism $f\colon A\rightarrow\C$, we obtain the $\C$-algebra $\mathcal{H}^f:=\C\otimes_A\mathcal{H}$ with basis $\{1\otimes a_w\colon w\in W(\lambda)\}$.  By \cite[Proposition 4.7]{howlettlehrer83}, when $\mathcal{H}^f$ is semisimple, this yields a bijection $\eta\mapsto \eta^f$ between $K$-characters of simple $\mathcal{H}^K$-modules and characters of simple $\mathcal{H}^f$-modules, where $\eta^f(1\otimes a_w)=f(\eta(a_w))$. 

In particular, the morphisms $f$ and $g \colon A\rightarrow \C$ will denote extensions of $f_0, g_0\colon A_0\rightarrow\C$ given by $f_0(u_\alpha)=p_{\alpha,\lambda}$ and $g_0(u_\alpha)=1$, respectively.  By  \cite[Lemma 4.2]{howlettlehrer83}, these extensions exist and yield isomorphisms $\mathcal{H}^f\cong \End_G(\fl\la)$ and $\mathcal{H}^g\cong \C W(\lambda)$ via $1\otimes a_w\mapsto T_{w,\lambda}$ and $1\otimes a_w\mapsto w,$ respectively.

\subsubsection{Further Remarks on the Generic Algebra in the Case of the Principal Series}\label{sec:genericsubalg}

For this section only, we assume that $L=T$ , so that $\lambda\in\irr(T)$ is a linear character and we are in the case of the principal series.  We fix here some notation for later use.

Let $\mathcal{H}_0$ denote the subalgebra of $\mathcal{H}$ generated by $\{a_w\mid w\in R(\lambda)\}$.  Then $\mathcal{H}_0$ is a generic algebra corresponding to the Coxeter group $R(\lambda)$.  In later sections, we will be interested in understanding the analogy between the Clifford theory from $R(\lambda)$ to $W(\lambda)$ and the characters of $\mathcal{H}_0$ in relation to $\mathcal{H}$.  Here we record some of the essentials.

First, by \cite[Theorem 3.7]{HowlettKilmoyer}, any irreducible character $\psi$ of $\mathcal{H}_0^K$ extends to an irreducible character $\tau$ of $C(\lambda)_{\psi}\mathcal{H}_0^K$, where $C(\lambda)_{\psi}\mathcal{H}_0^K$ is the tensor product $KC(\lambda)_\psi\otimes\mathcal{H}_0^K$ as defined in \cite{HowlettKilmoyer}.  Further, by \cite[(3.11) and Lemma 3.12]{HowlettKilmoyer}, every irreducible character of $\mathcal{H}^K$ is of the form $\tau^{\mathcal{H}}$, where $\tau$ is an extension of some $\psi\in\irr(\mathcal{H}_0^K)$ to $C(\lambda)_{\psi}\mathcal{H}_0^K$ and 
\[\tau^\mathcal{H}(a_w):=|C(\lambda)_\psi|^{-1}\sum \tau(a_{cwc^{-1}})\] where the sum is over $c\in C(\lambda)$ satisfying $cwc^{-1}\in C(\lambda)_\psi R(\lambda)$.

\subsection{Alternate Twists and Action on the Harish-Chandra Series}

Recall that the groups $W(\lambda)$ and $W(\la^\sigma)$ are the same, as $N(L)_\la=N(L)_{\la^\sigma}$.  We also see that $\e{\la}$ and $\e{\la^\sigma}$ are isomorphic as vector spaces, simply by mapping each basis element $T_{w,\la}$ to the corresponding basis element $T_{w,\la^\sigma}$ and extending linearly.  However, recalling that the structure constants for $\e{\la}$ and $\e{\la^\sigma}$ are the same (depending only on the numbers $p_{\alpha,\la}=p_{\alpha,\la^\sigma}$), we see that in fact, this yields an isomorphism of $\F$-algebras.
\begin{lemma}
The map $\e{\la}\rightarrow \e{\la^\sigma}$, defined by $T_{w,\la}\mapsto T_{w,\la^\sigma}$ and extending linearly, is an isomorphism of $\F$-algebras.
\end{lemma}

With this in place, given a matrix representation $\mathfrak{X}\colon \e{\la}\rightarrow \mathrm{Mat}_d(\F)$ for $\e{\la}$, we immediately obtain a corresponding representation \[\overline{\mathfrak{X}}\colon\e{\la^\sigma}\rightarrow \mathrm{Mat}_d(\F)\] via $\overline{\mathfrak{X}}(T_{w,\la^\sigma})=\mathfrak{X}(T_{w,\la})$.  Further, since the structure constants are rational, we obtain a representation
\[\mathfrak{X}^{(\sigma)}\colon\e{\la}\rightarrow \mathrm{Mat}_d(\F)\] defined by $\mathfrak{X}^{(\sigma)}(T_{w,\la})=\mathfrak{X}(T_{w,\la})^\sigma$, and extending linearly.  Together, this also yields a representation $\overline{\mathfrak{X}}^{(\sigma)}$ of $\e{\la^\sigma}$ satisfying that $\overline{\mathfrak{X}}^{(\sigma)}(T_{w,\la^\sigma})=\mathfrak{X}(T_{w,\la})^\sigma$.

If $\eta$ is a character of $\e{\la}$ afforded by $\mathfrak{X}$, we denote by $\overline{\eta}, \eta^{(\sigma)},$ and $\overline{\eta}^{(\sigma)}$ the corresponding characters afforded by $\overline{\mathfrak{X}}, \mathfrak{X}^{(\sigma)},$ and $\overline{\mathfrak{X}}^{(\sigma)}$, respectively.  Note then that by \prettyref{prop:moduleact}, 
\[\overline{\eta}^{(\sigma)}(T_{w,\la^\sigma}) = \eta(T_{w,\la})^\sigma={r}_{\sigma}(w)\delta'_{\lambda,\sigma}(w)\eta^\sigma(T_{w,\la^\sigma}).\]

Now let $\mathfrak{f}\colon\irr(\e{\la})\rightarrow\irr(W(\la))$ denote the bijection induced by the standard specializations of the generic algebra discussed above.  Taking into consideration \prettyref{prop:moduleact}, we are interested in how this bijection behaves under the action of $\sigma$ on the values of $\eta\in\irr(\e{\la})$.  It is clear that this bijection does not preserve fields of values, for example from the fact that the field of values for a Weyl group of type $E_8$ is $\Q$ but for the Hecke algebra $\End_G(1_B^G)$ is $\Q(\sqrt{q})$.  

Hence we define a non-standard bijection $\irr(W(\lambda))\rightarrow\irr(W(\la^\sigma))$ induced by $\sigma$ as follows.  We denote by $\gamma^{(\sigma)}$ the irreducible character of $W(\lambda^\sigma)$ such that $\gamma^{(\sigma)}=\mathfrak{f}\left(\overline{\eta}^{(\sigma)}\right)$, where $\eta\in\irr(\e{\la})$ is the character such that $\gamma=\mathfrak{f}(\eta)$.  That is, this new action of $\sigma$ on $\irr(W(\lambda))$ is such that 
\[\left(\mathfrak{f}(\eta)\right)^{(\sigma)}=\mathfrak{f}(\overline{\eta}^{(\sigma)}).\]  

\prettyref{prop:moduleact}, together with the above discussion, yields a description of how the Galois group acts on constituents of $R_L^G(\lambda)$, which is the Galois action analogue to \cite[Theorem 4.6]{MalleSpathMcKay2}.  (Note that since $\Q\subseteq\F\subseteq\C$ is a splitting field, the complex characters of $W(\lambda), \End_{\C G}(\fl{\la}\otimes_{\F}\C),$ and $G$ can be considered naturally as $\F$-characters of $W(\lambda), \End_{\F G}(\fl{\la}),$ and $G$.)  We write $R_L^G(\lambda)_\gamma$ to denote the constituent of $R_L^G(\lambda)$ afforded by a module $S$ such that $\Hom_{\F G}(\mathfrak{F}(\la), S)$ affords $\eta$ and $\mathfrak{f}(\eta)=\gamma$.


\begin{theorem}\label{thm:GaloisAct}
Let $\sigma\in\aut(\F)$ and let $\gamma\in \irr(W(\la))$.  Then

\[\left(R_L^G(\la)_\gamma\right)^\sigma=R_L^G(\lambda^\sigma)_{\gamma'},\] where $\gamma'\in\irr(W(\la))=\irr(W(\la^\sigma))$ is defined by $\gamma'(w)=r_\sigma(w)\delta'_{\la,\sigma}(w^{-1})\gamma^{(\sigma)}(w)$ for each $w\in W(\lambda)$.
\end{theorem}

We remark that when $R(\la)\leq\ker(\delta_{\la,\sigma})$, $\gamma'$ is just the character $\gamma'=r_\sigma\delta_{\la,\sigma}^{-1}\gamma^{(\sigma)}$.

\section{The Proof of Theorem A 
}\label{sec:SN2S}

As an application of our analysis in \prettyref{sec:gallact}, we complete the proof of Navarro's self-normalizing Sylow $2$-subgroup conjecture (see \prettyref{conj:mainprob}) 
by proving that the remaining simple groups are SN2S-Good, in the sense described in \cite{SchaefferFrySN2S1} (see also below).   We now finally specialize our choice of $\sigma$ to be the Galois automorphism relevant for \prettyref{conj:mainprob}.

\begin{center}
\fbox{
\parbox{6in}{\begin{center}
For the remainder of the article, $\sigma$ will always denote the Galois automorphism in \prettyref{conj:mainprob}, and $\K$ will always denote the fixed field of $\sigma$.\end{center}}
}
\end{center}

\subsection{The Conditions for Simple Groups}\label{sec:SimpleStatements}
Here we recall the definitions and main reduction theorem from \cite{SchaefferFrySN2S1}.

\begin{condition}\label{cond:conjIF} 
Let $G$ be a finite quasisimple group with center $Z:=Z(G)$ and $Q$ a finite $2$-group acting on $G$ as automorphisms.  Assume $P/Z\in \syl_2(G/Z)$ is $Q$-invariant and $C_{N_G(P)/P}(Q)=1$.  Then for any $Q$-invariant, $\sigma$-fixed $\lambda\in\irr(Z)$, we have $\chi^\sigma=\chi$ for any $Q$-invariant $\chi\in\irr_{2'}(G|\lambda)$.
\end{condition}
We note that the condition $C_{N_G(P)/P}(Q)=1$ is equivalent to $GQ/Z$ having a self-normalizing Sylow $2$-subgroup (see, for example, \cite[Lemma 2.1 (ii)]{NavarroTiepTurull2007}) and that to prove \prettyref{cond:conjIF}, it suffices to prove the statement for the Schur cover of $G/Z$.  

\begin{condition}\label{cond:conjFI} 
Let $S$ be a finite nonabelian simple group, $Q$ a finite $2$-group acting on $S$ as automorphisms, and $P\in \syl_2(S)$ be $Q$-invariant.  If every $Q$-invariant $\chi\in\irr_{2'}(S)$ is fixed by $\sigma$, then $C_{N_S(P)/P}(Q)=1$.
\end{condition}

\begin{definition}{\cite[Definition 1]{SchaefferFrySN2S1}}\label{def:Goodness}
Let $S$ be a finite nonabelian simple group.  We will say $S$ is ``SN2S-Good" if $S$ satisfies \prettyref{cond:conjFI} and $G$ satisfies \prettyref{cond:conjIF} whenever $G$ is a quasisimple group with $G/Z(G)\cong S$.

\end{definition}

\begin{theorem}{\cite[Theorem 3.7]{SchaefferFrySN2S1}}\label{thm:NTT62anal} 
 Let $G$ be a finite group and $P\in\syl_2(G)$.  Assume that every finite nonabelian simple group involved in $G$ is SN2S-Good (see \prettyref{def:Goodness}). Then $P=N_G(P)$ if and only if every $\chi\in\irr_{2'}(G)$ is fixed by $\sigma$.
\end{theorem}

The results of \cite{SchaefferFrySN2S1} and \cite{SFTaylorTypeA} imply that to prove \prettyref{conj:mainprob}, it suffices to show that when $q$ is odd, the simple groups $E_6^\pm(q)$, $E_8(q)$, $P\Omega_n^\pm(q)$ for $n\geq 7$, and $PSp_{2n}(q)$ for $n\geq2$ are SN2S-Good.
Here $E_6^+(q)$ means the untwisted simple group $E_6(q)$, $E_6^-(q)$ denotes the twisted simple group $\tw{2}{E}_6(q),$ and $P\Omega_n^\pm(q)$ is taken to mean $P\Omega_{n}(q)$ in the case $n$ is odd.
 
Sylow $2$-normalizers for finite simple groups are considered in \cite{kondratiev2005}.  From this, we see that if $S$ is one of the simple groups $PSp_{2n}(q)$ with $q\equiv\pm1\pmod8$, $P\Omega_n^\pm(q)$ with $q$ odd, or $E_8(q)$ with $q$ odd, then $S$ has a self-normalizing Sylow 2-subgroup, and hence this will hold for the Schur cover $G$ as well.  In these cases, to prove SN2S-Goodness, it therefore suffices to show that every irreducible character of $G$ of odd degree is fixed by $\sigma$. 

 If $S$ is $PSp_{2n}(q)$ with $q\equiv\pm3\mod8$, there is not a self-normalizing Sylow $2$-subgroup.  Further, in this case $q$ must be an odd power of an odd prime, so a $2$-group of automorphisms $Q$ must be contained in $PCSp_{2n}(q)$, which does have a self-normalizing Sylow 2-subgroup (see, for example, \prettyref{lem:CSpSN2S} below). Further, \cite[Proposition 4.8]{Malle08} and the discussion preceding it yields that the $PCSp_{2n}(q)$-invariant characters with odd degree are exactly the unipotent characters of odd degree and that since the Schur multiplier is size $2$, odd-degree characters of the Schur cover can be considered as characters of $S$. Thus in this case it suffices to show that every unipotent character of odd degree of $S$ is fixed by $\sigma$ and that there exists some non-unipotent $\chi\in\irr_{2'}(S)$ which is not fixed by $\sigma$.  The first of these conditions is verified by the following proposition.

\begin{proposition}\label{prop:unipsfixed}
Let $G$ be a finite group of Lie type with odd defining characteristic and no component of Suzuki or Ree type.  Then every unipotent character of odd degree is realizable over $\Q$.
\end{proposition}
\begin{proof}
By \cite[Proposition 7.4]{MalleSpathMcKay2}, every unipotent character of odd degree lies in the principal series of $G$.  But according to \cite[Theorem 2.9]{BensonCurtis} (see also \cite{BensonCurtisCorrection}), every member of $\irr_{2'}(G|R_T^G(1))$ can be realized over $\Q$.
\end{proof}

\subsection{Toward SN2S-Goodness}\label{sec:oddchars}

We collect here some additional observations that will be useful for determining SN2S-Goodness.  In what follows, we let $\bG{G}$ be a connected reductive algebraic group and $G=\textbf{G}^F$ be the group of fixed points of $\textbf{G}$ under a Frobenius endomorphism $F$ defined over $\F_q$.  We will at times do computations in $G$ using the Chevalley generators and relations, as in \cite[Theorem 1.12.1]{gorensteinlyonssolomonIII}.  In particular, $x_\alpha(t), n_\alpha(t),$ and $h_\alpha(t)$ are as defined there.

The dual of $G$ is $G^\ast=\textbf{G}^{\ast F^\ast}$, where $\bG{G}^\ast$ is dual to $\bG{G}$ as in, for example, \cite[Section 4.2]{carter2}.  Let $\textbf{T}$ and $\textbf{B}$ be an $F$-stable maximal torus and Borel subgroup of $\textbf{G}$, respectively, with $\textbf{T}\leq\textbf{B}$  fixed as in \cite[Section 2.B]{MalleSpathMcKay2} and let $\bG{T}^\ast$ be an $F^\ast$-stable maximal torus of $\bG{G}^\ast$ dual to $\bG{T}$.  Write $T=\textbf{T}^F$ and $T^\ast=(\bG{T}^\ast)^{F^\ast}$.   We write $W=\bg{W}^F$, where $\bg{W}=N_\bg{G}(\bg{T})/\bg{T}$, and similarly for $W^\ast$.

This duality induces an isomorphism $\irr(T)\rightarrow T^\ast$.  We begin by reconciling our understanding of $W(\la)$ in the case $\la\in\irr(T)$ with a Weyl group, using this isomorphism.  We remark that given $s\in T^\ast$, the only nontrivial action of $F^\ast$ on the Weyl group $\bg{W}(s)$ of $C_{\bG{G}^\ast}(s)$ is given by the automorphism induced by that of $F^\ast$ on $\bg{W}^\ast$, since $s$ is contained in a maximally split torus.

\begin{lemma}\label{lem:W(s)}

Let $\la\in\irr(T)$ and let $s\in T^\ast$ correspond to $\lambda$ in the sense of \cite[Proposition 4.4.1]{carter2}.  Denote by $W(s)$ and $W^\circ(s)$ the fixed points of the Weyl groups of $C_{\bG{G}^\ast}(s)$ and  $C_{\bG{G}^\ast}^\circ(s)$, respectively, under $F^\ast$.  Then 
\begin{enumerate}[label={(\arabic*)}]
\item $W(\la)$ is isomorphic to $W(s)$. 
\item If $\bG{G}$ is simple of simply connected type, not of type $A_n$, then there is an isomorphism $\kappa\colon W(\la)\rightarrow W(s)$ such that $\kappa(R(\la))=W^\circ(s)$.  In particular, in this case $W(\la)/R(\la)$ is isomorphic to $(C_{\bG{G}^\ast}(s)/C_{\bG{G}^\ast}^\circ(s))^{F^\ast}$.

\end{enumerate}

\end{lemma}
\begin{proof}
Let $W^\ast$ denote the Weyl group of $\bG{G}^\ast$ with respect to $\bG{T}^\ast$ and write $(X(\bG{T}), \Phi, Y(\bG{T}), \Phi^\ast)$ and $(X(\bG{T}^\ast), \Phi^\ast, Y(\bG{T}^\ast), \Phi)$ for the root data for $\bG{G}$ and $\bG{G}^\ast$, respectively.  By the duality of $\bG{G}$ and $\bG{G}^\ast$, we have an isomorphism $X(\bG{T})\rightarrow Y(\bG{T}^\ast)$, which induces an isomorphism $W\rightarrow W^\ast$.  Statement (1) follows from the descriptions of this isomorphism and the isomorphism $\irr(T)\cong T^\ast$ (see, for example, \cite[Propositions 4.2.3, 4.4.1]{carter2}), together with the identification of $W(s)$ as the subgroup of $W^\ast$ consisting of the elements which fix $s$ and that of $W(\la)$ as the subgroup of $W$ of elements which fix $\la$.

Now let $\bG{G}$ be simple of simply connected type, not of type $A_n$.  Then by \cite[Lemma 5.1]{MalleSpathMcKay2} and its proof, we have $R(\la)$ is generated by the $s_\alpha$ for $\alpha\in\Phi$ such that $\la(h_{\alpha}(t))=1$ for each $t\in\F_q^\times$.  But we may identify $W^\circ(s)$ as the subgroup of $W^\ast$ generated by $s_{\alpha^\ast}$, $\alpha^\ast\in\Phi^\ast$ such that $\alpha^\ast(s)=1$.  Hence statement (2) again follows from the definitions of the isomorphisms and \cite[Remark 2.4]{dignemichel}.  
 \end{proof}

It is clear that an understanding of odd-degree characters will be imperative for proving SN2S-Goodness.  For this reason, we next recall some key statements regarding odd-degree characters from \cite{MalleSpathMcKay2}.   

\begin{theorem}{\cite[Theorem 7.7]{MalleSpathMcKay2}}\label{thm:MS7.7}
Let $\textbf{G}$ be simple, of simply connected type, not of type ${A}_n$.  Let $\chi \in\irr_{2'}(G)$. Then either $\chi$ lies in the principal series of
$G$, or $q\equiv 3\pmod{4}$, $G=Sp_{2n}(q)$ with $n\geq 1$ odd, $\chi\in\mathcal{E}(G,s)$ with $C_{G^\ast}(s)=B_{2k}(q)\cdot\tw{2}{D}_{n-2k}(q).2$, where $0\leq k\leq (n-3)/2$ and $\chi$ lies in the Harish-Chandra series of a cuspidal character of degree $(q-1)/2$ of a Levi subgroup $Sp_2(q)\times (q-1)^{n-1}$.\end{theorem}

\begin{lemma}{\cite[extension of Lemma 7.9]{MalleSpathMcKay2}}\label{lem:MS7.9}
 Let $\textbf{G}$ be simple, of simply connected type, not of type ${A}_n$. Let $\chi\in\irr_{2'}(G)$.  Then $\chi=R_T^G(\lambda)_\gamma$, where $T$ is a maximally split torus of $G$, $\lambda\in\irr(T)$ is such that $2\nmid[W:W(\lambda)]$, and $\gamma\in\irr_{2'}(W(\lambda))$, except possibly in the case  $\textbf{G}$ is type ${C}_n$ and $q\equiv3\pmod{4}$.  In the latter case, $\chi$ may also be of the form $\chi=R_L^G(\lambda)_{\gamma}$ with $(L,\lambda)$ as in \prettyref{thm:MS7.7}, $2\nmid[W(L):W(\lambda)]$ and $\gamma\in\irr_{2'}(W(\lambda))$.
\end{lemma}

The following lemma will be useful for verifying that a rational Lusztig series of $G$ containing an odd-degree character is fixed by $\sigma$.
\begin{lemma}\label{lem:lambdafixseriesfix}
Let $\chi$ be a constituent of $R_T^G(\lambda)$ where $T$ is a maximally split torus of $G$ and $\lambda\in \irr(T)$.  If $\lambda^\sigma=\lambda$, then $\chi$ lies in a rational Lusztig series $\mathcal{E}(G,s)$ indexed by some $s\in G^\ast$ such that $\mathcal{E}(G,s)^\sigma=\mathcal{E}(G,s)$.
\end{lemma}
\begin{proof}
Let $\chi\in\mathcal{E}(G,s)$.  By \cite[11.10]{bonnafe06}, we know $\mathcal{E}(G,s)$ is a disjoint union of Harish-Chandra series $R_L^G(\lambda)$ with $\lambda\in\mathcal{E}(L,s)$.  If $\chi$ is a member of $R_T^G(\lambda)$ with $\lambda^\sigma=\lambda$, this forces $\mathcal{E}(T,s)^\sigma=\mathcal{E}(T,s)$.  Writing $s=s_2s_{2'}$ where $s_2$ is a $2$-element and $s_{2'}$ is a $2'$-element, this means that $s$ is conjugate to $s_2s_{2'}^2$ in $T^\ast$, and hence $G^\ast$, by \cite[Lemma 3.4]{SFTaylorTypeA}.  This yields that $\mathcal{E}(G,s)^\sigma=\mathcal{E}(G,s)$, again by \cite[Lemma 3.4]{SFTaylorTypeA}.
\end{proof}

\subsubsection{On the Term $r_\sigma$}

Recall that the term $r_\sigma(w)$ appears in \prettyref{thm:GaloisAct}, where writing $w=w_1w_2\in W(\la)$ for $w_1\in C(\lambda)$ and $w_2\in R(\lambda)$, we have  $r_\sigma(w):={\sqrt{\ind(w_1)}^\sigma}/{\sqrt{\ind(w_1)}}$. For the purpose of the next lemma only, we relax our assumptions on $G$ and return to the more general setting of \prettyref{sec:gallact}, in the sense that $G$ is not necessarily a group of Lie type.  (However, $\sigma$ is still the Galois automorphism as in \prettyref{conj:mainprob}.)

Using Gauss sums, ${\sqrt{p}=\sum_{n=1}^{p-1}{\legendre{n}{p}}\zeta_p^{n}}$ or ${-\sqrt{-1}\sum_{n=1}^{p-1}{\legendre{n}{p}}\zeta_p^{n}}$ for an odd prime $p$, where ${\legendre{n}{p}}$ denotes the Legendre symbol and $\zeta_p$ is a primitive $p$th root of unity.  Then ${\sqrt{p}^\sigma=\sum_{n=1}^{p-1}{\legendre{n}{p}}\zeta_p^{2n}}$ or ${-\sqrt{-1}\sum_{n=1}^{p-1}{\legendre{n}{p}}\zeta_p^{2n}}$.  But note that ${\legendre{2n}{p}}={\legendre{2}{p}}{\legendre{n}{p}}$ and ${\legendre{2}{p}}=1$ if $p\equiv\pm1\mod 8$ and $-1$ if $p\equiv \pm3\mod 8$.   This proves the following lemma:

\begin{lemma}\label{lem:rwgen}
Let $w\in W(\lambda)$ and let $r_\sigma$ be defined as above.  Then
\begin{itemize}
\item If $\ind{(w_1)}\equiv\pm1\mod8$, then $r_\sigma(w)=1$.   
\item If $\ind(w_1)\equiv\pm3\mod8$, then $r_\sigma(w)=-1$.
\end{itemize}
\end{lemma}

Now, and for the remainder of the article, we assume $G=\textbf{G}^F$ is a group of Lie type defined over $\F_q$.  By  \cite[Section 2.9]{carter2}, $\ind(w_1)=q^{\ell(w_1)}$, where $\ell(w_1)$ is the length of $w_1$ in the Weyl group of $\textbf{G}$.    This yields that $r_\sigma$ is a character in this case, and that we have the following reformulation of \prettyref{lem:rwgen} in the case of groups of Lie type:
\begin{lemma}\label{lem:rw}
Let $q$ be a power of an odd prime and let $G=\textbf{G}^F$ be a group of Lie type defined over $\F_q$.  For $w\in W(\lambda)$, \begin{itemize}
\item If $q\equiv\pm1\mod8$, then $r_\sigma(w)=1$.  
\item If $q\equiv\pm3\mod8$, then $r_\sigma(w)=(-1)^{\ell(w_1)}$, where $\ell(w_1)$ is the length of $w_1$ in the Weyl group of $\textbf{G}$.

\end{itemize}

 \end{lemma}

We next address the case of the values of $r_\sigma$ when $G=\textbf{G}^F$ is a group of Lie type $B_n$.

\begin{lemma}\label{lem:ClatypeB}
Let $G=\textbf{G}^F$ be of simply connected type, where $\textbf{G}$ is of type ${B}_n$ ($n\geq 3$).  Let $\chi\in\irr_{2'}(G)$, and further write $\chi=R_T^G(\lambda)_\gamma$ as in \prettyref{lem:MS7.9}.  Then every member of $C(\lambda)$ has even length in the Weyl group of $\textbf{G}$.
 \end{lemma}
\begin{proof}
Recall that by \prettyref{lem:MS7.9}, the index $[{W}:W(\lambda)]$ is odd, so that $\lambda$ is fixed by a Sylow $2$-subgroup of ${W}$.  Here $W$ can be viewed as a wreath product $C_2\wr\mathfrak{S}_n=\overline{K}\rtimes \mathfrak{S}_n$ with $\overline{K}\cong C_2^n$.  In particular, note that this means $\overline{K}$ is contained in $W(\la)$.  

Further, we fix a regular embedding $\bG{G}\hookrightarrow \wt{\bG{G}}$ as in \cite[Section 15.1]{cabanesenguehard}, so that $\wt{\bG{G}}$ has connected center, and write $\wt{G}=\wt{\bG{G}}^F$.  Recall that this induces a surjection $\wt{\bG{G}}^\ast\rightarrow\bG{G}^\ast$.  Let $\wt{\bG{T}}$ be an $F$-stable maximal torus of $\wt{\bG{G}}$ containing $\bG{T}$, let $\wt{\bG{T}}^\ast$ be dual to $\wt{\bG{T}}$, and denote by $\wt{T}=\wt{\bG{T}}^F$ and $\wt{T}^\ast=(\wt{\bG{T}}^\ast)^{F^\ast}$ the corresponding tori of $\wt{G}$ and $\wt{G}^\ast$.

Then we may extend $\lambda$ to a character $\wt{\lambda}$ of $\wt{T}$.  Let $s$ and $\wt{s}$ be semisimple elements in $T^\ast$ and $\wt{T}^\ast$, respectively, corresponding to $\lambda$ and $\wt{\la}$.  Then by \prettyref{lem:W(s)}, $R(\la)\cong W^\circ(s)\cong W(\wt{s})\cong W(\wt{\la})$ and $[W:R(\la)]_2\leq 2$.  This yields that $R(\la)$ must be a direct product of the form $C_2\wr \mathfrak{S}_{n_1}\times C_2\wr \mathfrak{S}_{n_2} \times \mathfrak{S}_{n_3}$, where $n=n_1+n_2+n_3$ and $n_3\leq 1$.  We also note that $C(\la)$ induces a graph automorphism on $R(\la)$, so $n_1=n_2$ if $C(\la)$ is nontrivial.

If $n_3=1$, we must have $[\mathfrak{S}_n:\mathfrak{S}_{n_1}\times \mathfrak{S}_{n_2}]$ is odd and $n_1+n_2=n-1$.  But this cannot occur if $n_1=n_2$, and hence in this case $R(\la)=W(\la)$.  However, note that this contradicts the assumption that $W(\la)$ contains $\overline{K}$.

Hence we see that $R(\la)$ is of the form $C_2\wr \mathfrak{S}_{n_1}\times C_2\wr \mathfrak{S}_{n_2}$ with $n_1+n_2=n$.  The fact that $[\mathfrak{S}_n:\mathfrak{S}_{n_1}\times \mathfrak{S}_{n_2}]_2\leq 2$ forces either $R(\lambda)=W(\lambda)$ or $n_1=n_2$ is a $2$-power.  (Recall that we exclude the case $n=2$.)  But this implies that the elements mapping $\mathfrak{S}_{n_1}$ to $\mathfrak{S}_{n_2}$ have even length, as desired.
\end{proof}

Next, we consider the case that $G$ is of type $D_n^\pm$.  For this, we recall the embedding of the simply connected group of type ${D}_n$ into ${B}_n$ discussed in \cite[2.C]{MalleSpathMcKay2} and \cite[Sections 10-12]{Spath10}.  

Write $\overline{\bf{G}}$ for the simple algebraic group of simply connected type over $\overline{\F}_q$ with root system $\overline{\Phi}=\{\pm e_i\pm e_j\}\cup \{\pm e_i\}$ of type $B_n$, where $\{e_1,...,e_n\}$ is an orthonormal basis for a Euclidean space.  The root system $\overline{\Phi}$ has simple roots $\overline{\Delta} = \{\overline{\alpha_1}, \alpha_2,...,\alpha_n\}$ with $\overline{\alpha}_1=e_1$ and $\alpha_i=e_i-e_{i-1}$ for $i>1$.  Then there is a root subsystem $\Phi\subseteq\overline{\Phi}$ of type $D_n$ comprised of all the long roots, with simple roots $\Delta = \{\alpha_1,...,\alpha_n\}$, where $\alpha_1=2\overline{\alpha}_1+\alpha_2$.  We write $\textbf{G}$ for the simply connected group corresponding to $\Phi$.  Letting $\textbf{T}$ and $\overline{\textbf{T}}$ be the corresponding maximally split tori and $\textbf{N}$ and $\overline{\textbf{N}}$ their normalizers in $\textbf{G}$ and $\overline{\textbf{G}}$, respectively, we have $\textbf{T}=\overline{\textbf{T}}$ and $\textbf{N}=\overline{\textbf{N}}\cap\textbf{G}$.  Note that then $T=\textbf{T}^F$ and $\overline{T}=\overline{\textbf{T}}^F$ coincide.  We write $\bG{W}$ and $\overline{\bG{W}}$ for the corresponding Weyl groups and $W=\bG{W}^F, \overline{W}=\overline{\bG{W}}^F$ for the Weyl groups of $G$ and $\overline{G}$, so that  $\bG{W}\leq \overline{\bG{W}}$ has index $2$.  

Note that $W\cong\bG{W}$ and $\overline{W}\cong\overline{\bG{W}}$ if $F$ is a split Frobenius endomorphism.  If $F$ is twisted, then $W\cong \overline{\bG{W}}'$, where $\overline{\bG{W}}'$ is a Weyl group of a root system of type $B_{n-1}$. For $\lambda\in\irr(T)$, we write $W(\lambda)=C(\lambda)\ltimes R(\lambda)$ and $\overline{W}(\lambda)=\overline{C}(\lambda)\ltimes \overline{R}(\lambda)$ for the decompositions of the corresponding relative inertia groups of $\lambda$ in $W$ and $\overline{W}$, respectively, as in \prettyref{sec:gallact}. 

Now, note that we may realize $\overline{\bG{W}}$ as the wreath product $C_2\wr \mathfrak{S}_n=\overline{K}\rtimes \mathfrak{S}_n$ with $\overline{K}=\overline{K}_1\times\cdots\times \overline{K}_n$ for $\overline{K}_i\cong C_2$.  Then ${\bG{W}}=K\rtimes \mathfrak{S}_n$, where $K$ is the subgroup of $\overline{K}$ of index two consisting of elements $(\epsilon_1,\ldots,\epsilon_n)$ for $\epsilon_i\in \{\pm1\}$ satisfying $\prod_{i=1}^n \epsilon_i=1$.  We write $k_i$ for the generator of $\overline{K}_i$, which can be seen to be induced by the element $n_{e_i}(1)$ of $\overline{\bf{N}}$, or $s_{e_i}$ in $\overline{\bG{W}}$.  

\begin{proposition}\label{prop:Clambdaeven}
Let $G=\textbf{G}^F$ be of simply connected type, where $\textbf{G}$ is of type ${B}_n$ ($n\geq 3$) or ${D}_n$ ($n\geq 4$) and $G\neq \tw{3}{D_4}(q)$.  Let $\chi\in\irr_{2'}(G)$, and further write $\chi=R_T^G(\lambda)_\gamma$ as in \prettyref{lem:MS7.9}.  Then every member of $C(\lambda)$ has even length in the Weyl group of $\textbf{G}$.
 \end{proposition}
\begin{proof}

For type $B_n$, this is just \prettyref{lem:ClatypeB}, so we assume $\textbf{G}$ is of type ${D}_n$.  

(1) Suppose first that $F$ is a split Frobenius endomorphism, so that $G$ is untwisted.  Note that since $T$ is maximally split, we may apply \cite[Lemma 11.3]{Spath10} in the case $\nu=1=\nu'$ to see that $W(\lambda)=\overline{W}(\lambda)\cap W$, and hence $[\overline{W}(\lambda):W(\lambda)]$ divides $2$.  Further, by \cite[Lemma 11.5]{Spath10}, $\overline{W}(\lambda)$ contains a normal subgroup of index dividing $2$ of the form $(A_1\times\cdots\times A_n)\rtimes U$ where $A_i\leq \overline{K}_i$ for each $i$ and $U$ is a product of symmetric groups. 

If $W(\lambda)=\overline{W}(\lambda)$, then since $[W:W(\lambda)]=[W:\overline{W}(\lambda)]$ is odd, we know $\overline{W}(\lambda)$ contains $K$.  Then $\overline{W}(\lambda)$ does not contain $k_i$ for any $i$, since then $\overline{W}(\lambda)$ would contain $\overline{K}$, contradicting $\overline{W}(\lambda)\leq W$.  It follows that each $A_i$ is trivial, so $\overline{W}(\lambda)$ contains $U$ with index at most $2$, contradicting $K\leq \overline{W}(\lambda)$. 

Hence we see $W(\lambda)$ has index $2$ in $\overline{W}(\lambda)$, so $\overline{W}=\overline{W}(\lambda)W$ and $[\overline{W}:\overline{W}(\lambda)]=[W:W(\lambda)]$ is odd, so $\overline{K}\leq \overline{W}(\lambda)$.  Further, using \cite[Lemma 5.1]{MalleSpathMcKay2}, we see that $\Phi_\lambda=\overline{\Phi}_\lambda\cap\Phi$.  This yields that $R(\lambda)=\overline{R}(\lambda)\cap W$, so $[\overline{R}(\lambda):R(\lambda)]$ divides $2$.  

Now, the proof of \prettyref{lem:ClatypeB} yields that $\overline{R}(\lambda)=\overline{W}(\lambda)$ or $\overline{R}(\lambda)\cong (C_2\wr\mathfrak{S}_{2^a})^2 $ where $n=2^{a+1}$ and $[\overline{W}(\lambda):\overline{R}(\lambda)]=2$.   This yields that $R(\lambda)=W(\lambda)$, unless $\overline{R}(\lambda)$ is as in the second case and $[\overline{R}(\lambda):R(\lambda)]=2$.  In this case, notice that $\overline{R}(\lambda)$ contains $\overline{K}$, so $R(\lambda)$ contains $\overline{K}\cap W=K$.  Then $R(\lambda)=K\rtimes (\mathfrak{S}_{2^a}\times\mathfrak{S}_{2^a})$, and $C(\lambda)=\overline{C}(\lambda)$ again consists of elements of even length.

(2) Finally, suppose that $F$ is twisted, so $W\cong \overline{\bG{W}}'$.  It will be useful to identify $W$ in two ways.  First, as a subgroup of $\bG{W}$, $W$ is generated by $s_2':=s_{\alpha_1}s_{\alpha_2}$ and $s_{\alpha_i}$ for $i\geq 3$.  Note that the generator $s_2'$ has even length in $\bG{W}$ and that it induces the graph automorphism on a subgroup of $W$ isomorphic to a Weyl group of type $D_{n-1}$.  Now, the map $s_2'\mapsto s_{e_2}; s_{\alpha_i}\mapsto s_{\alpha_i}, i\geq 3$ defines an isomorphism between $W$ and a Weyl group $\overline{\bG{W}}'$ of type $B_{n-1}$ with simple roots $\{\frac{1}{2}(\alpha_1+\alpha_2), \alpha_3,\ldots,\alpha_n\}=\{e_2,\alpha_3,\ldots,\alpha_n\}$.  We will write $\overline{\Phi}'$ for the corresponding root system, which is a subsystem of $\overline{\Phi}$.  Since $\overline{\bG{G}}$ is simply connected, we may apply \cite[Theorem 1.12.5]{gorensteinlyonssolomonIII} to $\overline{\bG{G}}':=\langle x_{\alpha}(t)\mid\alpha\in\overline{\Phi}'; t\in\overline{\F}_q\rangle$ and $\overline{\bG{T}}':=\langle h_\alpha(t)\mid\alpha\in\overline{\Phi}'; t\in\overline{\F}_q^\times\rangle$ to see that $\overline{\bG{G}}'$ is simply connected of type ${B}_{n-1}$.

Now, note that $T=\overline{\bG{T}}^F\cong \F_{q^2}^\times \times (\F_q^\times)^{n-2}$ is the set of elements of $\overline{\bG{T}}$ of the form $\prod_{i=1}^n h_{\alpha_i}(t_i)$ where $t_1, t_2\in\F_{q^2}^\times$ and $t_i\in\F_q^\times$ for $i\geq 3$, and $t_2=t_1^q$.  The subgroup $T'=\overline{\bG{T}}'^F\cong (\F_q^\times)^{n-1}$ of $T$ is the set of elements of $\overline{\bG{T}}$ of the form $\prod_{i=1}^n h_{\alpha_i}(t_i)$ where $t_i\in\F_q^\times$ for $i\geq 1$, and $t_2=t_1$.  Further, writing $N'=\overline{\bG{N}}'^F$ for $\overline{\bG{N}}'=\langle n_{\alpha}(t)\mid\alpha\in\overline{\Phi}'; t\in\overline{\F}_q^\times\rangle$, we have $N'T=N$ and $N'\cap T=T'$.  Then letting $\lambda'\in\irr(T'|\lambda)$, we have $W(\lambda)=N_{\lambda}/T$ is isomorphic to a subgroup of $\overline{W'} (\lambda')=N_{\lambda'}'/T'$, with odd index since $[W:W(\lambda)]$ is odd.  Writing $\overline{W}'(\lambda')=\overline{R}'(\lambda')\rtimes \overline{C}'(\lambda')$ for the corresponding decomposition, since $[\overline{\bG{W}}':\overline{W'}(\lambda')]$ is odd, we may again apply the proof of \prettyref{lem:ClatypeB} to see that $\overline{R'}(\lambda')$ is of the form $C_2\wr \mathfrak{S}_{n_1}\times C_2\wr \mathfrak{S}_{n_2}$ with $n_1+n_2=n-1$.  In particular, if we write $\overline{\bG{W}}$ as the wreath product $C_2\wr \mathfrak{S}_{n-1}=\overline{K}'\rtimes \mathfrak{S}_{n-1}$ in analogy to $\overline{\bG{W}}$, we see that $\overline{K}\leq \overline{R}'(\lambda')$ and $\overline{C}'(\lambda')$ is comprised of elements of even length.

Recall that $\overline{R}'(\lambda')=\langle s_{\alpha}\mid\alpha\in\overline{\Phi}'_{\lambda'}\rangle$, and viewing $W(\lambda)$ as a subgroup of $\overline{W}'(\lambda')$, we have $R(\lambda)=\langle s_{\alpha}\mid\alpha\in\overline{\Phi}'_{\lambda}\rangle$.  Note that by \cite[Lemma 5.1]{MalleSpathMcKay2}, $\overline{\Phi}'_\lambda=\{\alpha\in\overline{\Phi}'\mid\lambda(T\cap \langle X_\alpha, X_{-\alpha}\rangle) = 1\}$ and $\overline{\Phi}'_{\lambda'}=\{\alpha\in\overline{\Phi}'\mid\lambda'(T'\cap \langle X_\alpha, X_{-\alpha}\rangle) = 1\}$, so that $R(\lambda)$ is a subgroup of $\overline{R}'(\lambda')$.

Since $s_{e_2}\in \overline{R}'(\lambda')$, we see that either $s_{e_2}\in R(\lambda)$, in which case $\lambda$ is trivial on the $\F_{q^2}^\times$ factor of $T$, so $R(\lambda)=\overline{R}'(\lambda')$, or $s_{e_2}$ induces the quotient $\overline{R}'(\lambda')/R(\lambda)$.  Recalling that as a member of $\bG{W}$, $s_{e_2}$ corresponds to the element $s_{\alpha_1}s_{\alpha_2}$, we therefore see that the members of $C(\lambda)$ have even length in $\bG{W}$.
\end{proof}

\subsubsection{On the Character $\gamma^{(\sigma)}$}

\prettyref{thm:GaloisAct} and \prettyref{lem:MS7.9} also suggest that another important ingredient for SN2S-Goodness will be to understand the behavior of $\irr_{2'}(W(\la))$ under the action of $\sigma$.

\begin{proposition}\label{lem:W(la)}
Let ${G}$ be a group of Lie type with odd defining characteristic and let $\lambda\in\irr(T)$ such that $C(\lambda)$ is a $2$-group.  Then every $\gamma\in\irr_{2'}(W(\la))$ is fixed by $\sigma$.   
\end{proposition}
\begin{proof}
Let $\gamma\in\irr_{2'}(W(\la))$.  Recall that $W(\la)=R(\la)\rtimes C(\la)$ and that $R(\la)$ is a Weyl group.  By Clifford theory, $\gamma|_{R(\la)}=e\sum_{i=1}^t\theta^{g_i}$ for some $\theta\in\irr(R(\la))$ with $e, t$ dividing $|C(\la)|$.  Then since $\gamma$ has odd degree and $C(\la)$ is a $2$-group, we see that $e=t=1$ and $\gamma|_{R(\la)}=\theta\in\irr_{2'}(R(\la))$.  Now, since $R(\la)$ is a Weyl group, $\theta$ takes values in the rational numbers (see, for example, \cite[Theorems 5.3.8, 5.4.5, 5.5.6, and Corollary 5.6.4]{GeckPfeiffer}), so recalling that $C(\la)$ is abelian, we see by \cite[Lemma 3.4]{SchaefferFrySN2S1} that $\gamma$ is fixed by $\sigma$.  
\end{proof}

Due to the nature of the action of $\sigma$ on the parametrization $R_T^G(\lambda)_\gamma$ given in \prettyref{thm:GaloisAct}, what we really hope for is a corresponding statement to \prettyref{lem:W(la)} but for the character $\eta\in\End_G(\fl{\la})$ satisfying $\mathfrak{f}(\eta)=\gamma$.  We make progress toward this end with the next two observations.  Here we keep the notation of \prettyref{sec:genericsubalg}.

\begin{proposition}\label{prop:fixHf}
Let $\lambda\in\irr(T)$ such that $C(\lambda)$ is a $2$-group.  Let $\mathbb{K}$ denote the fixed field of $\sigma$, and suppose that every simple $\mathcal{H}_0^f$-module is afforded over $\mathbb{K}$.  Then $\eta^{(\sigma)}=\eta$ for every irreducible character $\eta$ of $\End_G(\fl{\la})$ of odd degree.  In particular, this is true for $\eta$ satisfying $\mathfrak{f}(\eta)=\gamma$ for $\gamma\in\irr_{2'}(W(\lambda))$.
\end{proposition}
\begin{proof}
Keeping the notation of \prettyref{sec:genericsubalg}, we recall that every irreducible character of $\mathcal{H}^f\cong\End_G(\fl{\la})$ is of the form $(\tau^{\mathcal{H}})_f$, where $\tau$ is an extension to $C(\lambda)_\psi\mathcal{H}_0^K$ of some irreducible character $\psi$ of the generic algebra $\mathcal{H}_0^K$ corresponding to $R(\lambda)$.  Observing the description of values of $\tau^{\mathcal{H}}$, notice that $\tau^{\mathcal{H}}(a_1)=[C(\lambda):C(\lambda)_{\psi}] \tau(a_1)$.  Hence, if $(\tau^{\mathcal{H}})_f$ has odd degree, then since $C(\lambda)$ is a $2$-group, we see that $C(\lambda)=C(\lambda)_{\psi}$, $\mathcal{H}^K=C(\lambda)\mathcal{H}_0^K=C(\lambda)_\psi\mathcal{H}_0^K$ and $\tau^\mathcal{H}=\tau$.  It therefore suffices to show that if $\psi$ has odd degree, then $\tau_f$ takes its values in $\mathbb{K}$. 

Recall that the characters of the form $\beta\tau$ for $\beta\in\irr(C(\lambda))$ form the complete set of extensions of $\psi$ to $\mathcal{H}^K$.  Let $\varrho$ be a representation of $\mathcal{H}_0^K$ affording $\psi$ and $\widetilde{\varrho}$ a representation of $\mathcal{H}^K$ affording $\tau$ and extending $\varrho$.  Notice that $\mu_0:=\det\circ \varrho$ is also a representation of $\mathcal{H}_0^K$.    As such, we see that $\mu_0$ also extends to a representation $\mu$ of $\mathcal{H}^K$, and the complete set of such extensions is given by $\beta\mu$ for $\beta\in\irr(C(\lambda))$.  In particular, we may take $\mu=\det\circ\widetilde{\varrho}$.

Now, from the proof of \cite[Theorem 3.8]{HowlettKilmoyer}, we see that $\mu(d\otimes b)=\mu(d)\mu(b)=\mu(d)\mu_0(b)$ for $d\in C(\lambda)$ and $b\in \mathcal{H}_0^K$.  Since $\mu$ is degree $1$, we may consider it as a character.  Then taking the specializations, we see that $\mu_f(d\otimes b)=\mu_f(d)(\mu_0)_f(b)$.  By assumption, $(\mu_0)_f(b)$ takes values in $\mathbb{K}$, and all characters of the $2$-group $C(\lambda)$ are $\sigma$-fixed.  Hence we see that in fact, $\mu_f$ takes values in $\mathbb{K}$.

Now, since $\psi_f$ has values in $\mathbb{K}$, we see that $(\tau_f)^{(\sigma)}$ must also be an extension of $\psi_f$.  Then $(\tau_f)^{(\sigma)}=\beta\tau_f$ for some $\beta\in\irr(C(\lambda))$, and hence $(\det\circ\widetilde{\varrho})_f=\mu_f=(\mu_f)^{(\sigma)}=(\det\circ\widetilde{\varrho}^{(\sigma)})_f=\beta^{\psi(1)}\mu_f$.  Since $\beta$ is a character of a $2$-group and $\psi(1)$ is odd, we see that $\beta=1$ and hence $\tau_f$ has values in $\mathbb{K}$, as desired.
\end{proof}

Armed with \prettyref{prop:fixHf}, we are interested in determining when the characters of  $\mathcal{H}_0^f$ are afforded over $\mathbb{K}$.  The next proposition gives us a partial answer to this question.

\begin{lemma}\label{lem:fixH0f}

 Let $\lambda\in\irr(T)$ and suppose that the Weyl group $R(\lambda)$ has no component of type $G_2, E_7,$ or $E_8$.  Then every simple $\mathcal{H}_0^f$ - module is realizable over $\Q$.
\end{lemma}
\begin{proof}
Let $R(\lambda)=W_1\times\cdots\times W_t$ be the decomposition of $R(\lambda)$ into irreducible Weyl groups.  Then we can write $\mathcal{H}_0=\mathcal{H}_1\otimes\cdots\otimes\mathcal{H}_t$ where $\mathcal{H}_i$ is the generic subalgebra generated by $\{a_w\mid w\in W_i\}$.  Suppose that $M$ is an irreducible $\mathcal{H}_0^f$-module which is not realizable over $\Q$.  Then we may write $M=M_1\otimes\cdots\otimes M_t$ where each $M_i$ is a simple $\mathcal{H}_i^f$-module.  Since $M$ is not realizable over $\Q$, there must be some $i$ such that $M_i$ is not realizable over $\Q$.  But this contradicts our assumption that $W_i$ is not of type $G_2,$ $E_7,$ or $E_8$, by \cite[9.3.4]{GeckPfeiffer}.
\end{proof}

We remark that the definition of $\mathfrak{f}$ and the fact that $W(\lambda)=W(\la^\sigma)$ yield that $\mathfrak{f}(\eta)=\mathfrak{f}(\overline{\eta})$.  Hence, if $\eta^{(\sigma)}=\eta$, then $\gamma^{(\sigma)}=\gamma$.  Then \prettyref{lem:fixH0f} and \prettyref{prop:fixHf} give the following:

\begin{corollary}\label{cor:newsigmaaction}
Let  $\lambda\in\irr(T)$ such that the Weyl group $R(\lambda)$ has no component of type $G_2, E_7,$ or $E_8$ and the complement $C(\lambda)$ is a $2$-group.  Then for $\gamma\in\irr_{2'}(W(\lambda))$, $\eta^{(\sigma)}=\eta$, where $\eta $ is the character of $\End_G(\fl\la)$ satisfying $\mathfrak{f}(\eta)=\gamma$.  Hence in this case, $\gamma^{(\sigma)}=\gamma$.

\end{corollary}

\subsection{SN2S-Goodness for $\bG{G}$ of Types ${B}_n, {C}_n, {D}_n,$ and ${E}_8$}

In this section, we show that for $q$ odd, the simple groups $PSp_{2n}(q)$ for $n\geq 2$, $P\Omega^\pm_{n}(q)$ for $n\geq 7$, and $E_8(q)$ are SN2S-Good.   

We begin with an observation about $\delta_{\lambda,\sigma}$ in the case $\lambda^\sigma=\lambda$ and $\lambda$ is linear, which follows from careful consideration of the extension maps constructed in \cite{Spath09} and \cite{MalleSpathMcKay2}.
\begin{proposition}\label{prop:extn}
Let $\bG{G}$ be simple, of simply connected type, not of type $A_n$, with odd defining characteristic.   Let $\la\in\irr_{2'}(L)$ be linear and fixed by $\sigma$.  Then there exists an extension $\Lambda_\la$ of $\la$ to $N(L)_{\la}$ which is $\sigma$-invariant.
\end{proposition}
\begin{proof}

Let $H$ and $V$ be defined as in \cite[Section 2]{Spath09}.  Then the proof of \cite[Lemma 4.2]{Spath09} yields that it suffices to note that there is a $\sigma$-fixed extension of $\la|_H$ to its stabilizer $V_\la$ in $V$, since $H$ is a $2$-group and $\lambda$ is linear.  
\end{proof}

\begin{theorem}\label{thm:princseries1mod8}
Let $G$ be a finite group of Lie type of simply connected type $B_n$ ($n\geq 3$), $C_n$ ($n\geq 2$), or $D_n^{\pm}$ ($n\geq4$) over a field with $q$ elements, where $q\equiv\pm1\mod 8$, or let $G$ be the simple group $E_8(q)$ for $q$ odd.  Then every $\chi\in\irr_{2'}(G)$ lying in the principal series is fixed by $\sigma$.
\end{theorem}

\begin{proof}
Let $\chi\in\irr_{2'}(G)$ lie in the principal series, so that there is $\lambda\in\irr(T)$ such that $\chi\in\irr_{2'}(G|R_T^G(\lambda))$.  By \prettyref{lem:MS7.9}, we may write $\chi=(R_T^G(\lambda))_\gamma$, where $\gamma\in\irr_{2'}(W(\la))$.

First assume that $G$ is type $B_n$, $C_n$, or $D^{\pm}_n$.  Then by \cite[Lemma 7.5]{MalleSpathMcKay2}, we see that $\lambda^2=1$, and hence $\lambda^\sigma=\lambda$.  Further, by \prettyref{prop:extn}, there is an extension $\Lambda_\la$ of $\la$ to $N_G(T)_{\la}$ which is $\sigma$-fixed, so that $\delta_{\la,\sigma}=1$ and certainly $R(\la)\leq\ker\delta_{\la,\sigma}$.   Further, note that $C(\lambda)$ embeds into $Z(\bG{G})/Z(\bG{G})^\circ$, yielding that $C(\lambda)$ is a $2$-group.  Then by \prettyref{thm:GaloisAct}, \prettyref{lem:rw}, and \prettyref{cor:newsigmaaction}, we have $\chi^\sigma=\chi$, using our assumption that $q\equiv\pm1\mod8$.  

Next assume that $G=E_8(q)$.  Then $G$ is self-dual and has a self-normalizing Sylow $2$-subgroup.  Hence any semisimple $s\in G^\ast$ centralizing a Sylow $2$-subgroup must have $2$-power order.  Then if $\chi\in\mathcal{E}(G,s)$, it follows that $\lambda$ also has $2$-power order so is fixed by $\sigma$, and by \cite[Theorem 3.2]{SchaefferFrySN2S1}, $\lambda$ is realizable over the field $\mathbb{K}$ of fixed points under $\sigma$.  Note that since $\gamma$ has odd degree, the corresponding character of $\e{\la}$ also has odd degree.  Then since $G$ is of adjoint type, \cite[Proposition 5.5]{Geck03} and its proof yield that $\chi$ can be realized over $\mathbb{K}$ as well.  (Indeed, note that the excluded characters in \cite[Proposition 5.5]{Geck03} correspond to characters of $\e{\la}$ of even degree.)
\end{proof}

When combined with \cite[Theorem 7.7]{MalleSpathMcKay2}, \prettyref{thm:princseries1mod8} yields the following immediate consequence.

\begin{corollary}\label{cor:princseries1mod8}
The simple groups $E_8(q)$ are SN2S-Good for all $q$.  The simple groups $P\Omega_n^\pm(q)$ are SN2S-Good for $n\geq 7$ when $q\equiv\pm1\mod8$.  Further, when $q\equiv1\mod8$, the simple groups $PSp_{2n}(q)$ for $n\geq 2$ are also SN2S-Good.
\end{corollary}

We next address the case that $S=PSp_{2n}(q)$.  The following observation yields that $PCSp_{2n}(q)$ has a self-normalizing Sylow $2$-subgroup.  Note that $PCSp_{2n}(q)$ is the group $\mathrm{InnDiag}(S)\leqslant\Aut(S)$, as defined in \cite[2.5.10(d)]{gorensteinlyonssolomonIII}.

\begin{lemma}\label{lem:CSpSN2S}
Let $S=PSp_{2n}(q)$ with $q\equiv\pm3\pmod 8$.  Then $\mathrm{InnDiag}(S)\cong PCSp_{2n}(q)$ has a self-normalizing Sylow $2$-subgroup.
\end{lemma}

\begin{proof}
Let $\wt{S}=PCSp_{2n}(q)$.  Note that $|\wt{S}/S|=2$, and in fact we may write $\wt{S}\cong S\rtimes  \langle\delta\rangle$ for some diagonal automorphism $\delta$.  Now, by \cite{carterfong}, we see that a Sylow $2$-subgroup of $G=Sp_{2n}(q)$ is the direct product $P=P_1\times\cdots\times P_t$, where $2n=2^{r_1}+\cdots +2^{r_t}$ is the 2-adic expansion of $2n$ and $P_i$ is a Sylow $2$-subgroup of $Sp_{2^{r_i}}(q)$.   Further, writing $\overline{P}=PZ(G)/Z(G)$ for a Sylow $2$-subgroup of $S$, we see that $N_G(P)/P\cong N_S(\overline{P})/\overline{P}$ is elementary abelian of order $3^t$, where the $i$th copy of the cyclic group $C_3$ corresponds to $N_{Sp_{2^{r_i}}(q)}(P_i)/P_i$.  Hence it suffices to show that $C_{V_i}(\delta)=1$, where we write $N_{Sp_{2^{r_i}}(q)}(P_i)=P_i\rtimes V_i$ with $V_i\cong C_3$.

Using the results of \cite{carterfong}, we can further realize $P_i$ as the wreath product $P'\wr T_{r_i-1}$, where $P'$ is a Sylow $2$-subgroup of $Sp_{2}(q)$ and $T_{r_i-1}$ is a Sylow $2$-subgroup of  the symmetric group $\mathfrak{S}_{2^{r_i}-1}$.  As such, $V_i$ acts on $P_i$ via the action of $N_{Sp_2(q)}(P')/P'\cong C_3$ on $P'$.  We also see that $\delta$ can be taken to act on $P'$ via the action of $PGL_2(q)/PSL_2(q)$, so that $C_{V_i}(\delta)=1$ since $PGL_2(q)$ has a self-normalizing Sylow $2$-subgroup, using \cite[Lemma 3]{carterfong}. 
\end{proof}

\begin{theorem}\label{thm:typeC}
Let $S=PSp_{2n}(q)$ for $n\geq 2$, where $q$ is odd.  Then $S$ is SN2S-Good.
\end{theorem}
\begin{proof}
Let $G=Sp_{2n}(q)$.  Note that by the results of \cite{SchaefferFrySN2S1}, we may assume that $S$ does not have an exceptional Schur multiplier, and if $q\equiv 1\mod 8$, we are done by \prettyref{cor:princseries1mod8}.  

First suppose $q\equiv 3\mod 4$, so that  \cite[Theorem 7.7]{MalleSpathMcKay2} yields that there are odd degree characters lying in a series $R_L^G(\lambda)$ where $L\cong Sp_2(q)\times T_1$ for $T_1\cong (q-1)^{n-1}$ and $\lambda=\psi\times\lambda_1$ with $\psi\in\irr(Sp_2(q))$ of degree $\frac{q-1}{2}$.  If $q\equiv 3\mod 8$, then $S$ does not have a self-normalizing Sylow $2$-subgroup.  The character values of $Sp_2(q)$ are well-known, and it is clear that $\psi^\sigma\neq\psi$, since $\sqrt{q}^\sigma\neq\sqrt{q}$, so $\lambda\in\irr(L)$ is not fixed by $\sigma$.  Hence $R_L^G(\lambda)$ is not fixed by $\sigma$, so there exist characters of odd degree not fixed by $\sigma$, and we are done in this case by \prettyref{prop:unipsfixed} and the discussion preceding it. If $q\equiv 7\mod 8$, then $G$ has a self-normalizing Sylow $2$-subgroup and it suffices to show that every $\chi\in\irr_{2'}(G)$ is fixed by $\sigma$.  In particular, \prettyref{thm:princseries1mod8} yields that it suffices to show that $\chi^\sigma=\chi$ for $\chi=R_L^G(\la)_\gamma$ in the non-principal series mentioned above.  Notice that $r_\sigma(w)=1$ by \prettyref{lem:rw}.  Further, since $\Lambda_\lambda$ is of the form $\psi\times\Lambda_1(\lambda_1)$ for an extension $\Lambda_1(\lambda_1)$ of $\lambda_1$ to $N(T_1)_{\lambda_1}$, we have $\delta_{\la,\sigma}=1$ by \prettyref{prop:extn}.  Since $W(\lambda)\cong W(\lambda_1)$, \prettyref{cor:newsigmaaction} still yields $\gamma^{(\sigma)}=\gamma$ in this situation, proving the statement by \prettyref{thm:GaloisAct}.

Finally, let $q\equiv 5\mod 8$.  As in the case $q\equiv 3\mod 8$, $S$ does not have a self-normalizing Sylow $2$-subgroup, and it suffices to show that there is a character of odd degree not fixed by $\sigma$.  Let $s\in G^\ast$ with $C_{G^\ast}(s)$ as in the first two lines of \cite[Table 1]{MalleSpathMcKay2} with $k$ the largest power of $2$ smaller than $n$.    Let $\{\alpha_1,...,\alpha_n\}$ denote the simple roots determined by $\bG{T}$ and $\bG{B}$, with $\alpha_1=e_1$ and $\alpha_i=e_i-e_{i-1}$ as in \cite[1.8.8]{gorensteinlyonssolomonIII}.  Let $\lambda$ be trivial on $ h_{e_i}(t)$ for $i> k$ and $t\in \F_{q}^\times$ and have order $2$ on the subgroups $\langle h_{e_i}(t)\colon t\in\F_q^\times\rangle$ for $1\leq i\leq k$.  Then $R(\lambda)$ is a reflection group of type $D_k\times B_{n-k}$, $W(\lambda)$ is of type $B_k\times B_{n-k}$, and $C(\lambda)$ induces the graph automorphism on the $D_k$ component.  Hence a nontrivial element in $C(\lambda)$ is conjugate to a simple reflection, meaning it has odd length in $W$, yielding that $r_\sigma(w)=-1$ by \prettyref{lem:rw}.  Then letting $\chi=R_T^G(\lambda)_\gamma$, we see that since $\lambda^\sigma=\lambda$ by construction, $\delta_{\la,\sigma}=1$ by \prettyref{prop:extn} and  $\gamma^{(\sigma)}=\gamma$ by \prettyref{cor:newsigmaaction}, it must be that $\chi^\sigma\neq\chi$ by \prettyref{thm:GaloisAct}.
\end{proof}

\begin{theorem}\label{thm:typeBD}
Let $S=P\Omega_{n}^{\pm}(q)$ for $n\geq 7$, where $q$ is odd.  Then $S$ is SN2S-Good.
\end{theorem}
\begin{proof}
Note that $S$ has a self-normalizing Sylow $2$-subgroup and that by the results of \cite{SchaefferFrySN2S1}, we may assume that $S$ does not have an exceptional Schur multiplier.  Let $G=\mathrm{Spin}_{n}^\pm(q)$.   Then by \prettyref{thm:GaloisAct}, \prettyref{prop:extn}, \prettyref{lem:rw}, \prettyref{cor:newsigmaaction}, and \prettyref{cor:princseries1mod8}, it suffices to note that by \prettyref{prop:Clambdaeven}, in the case $q\equiv\pm3\mod 8$, $\ell(w_1)$ is even for each $w_1\in C(\lambda)$ when $\chi=R_T^G(\lambda)_\gamma\in\irr_{2'}(G)$. 
\end{proof}

\subsection{SN2S-Goodness for $\bG{G}$ of Type ${E}_6$}

For this section, we keep the notation of the beginning of \prettyref{sec:oddchars}.  Let $q$ be a power of an odd prime $p$ and let $\bG{G}$ be simply connected of type ${E}_6$, so that $G=\bG{G}^F=E_6^\epsilon(q)_{sc}$, and let $S$ be the simple group $E_6^\epsilon(q)=G/Z$ with $Z:=Z(G)$.  Here we let $\epsilon=1$ (or $+$) if $F$ is split and $\epsilon = -1$ (or $-$) if $F$ is twisted.  Then the dual group $G^\ast\cong E_6^\epsilon(q)_{ad}$ satisfies $[G^\ast,G^\ast]\cong S$, and we make this identification.   Further, we fix a regular embedding $\bG{G}\hookrightarrow \wt{\bG{G}}$ and write $\wt{G}=\wt{\bG{G}}^F$ so that $G=[\wt{G}, \wt{G}]$ and $G^\ast\cong \wt{G}/Z(\wt{G})\cong \wt{G^\ast}/Z(\wt{G}^\ast)$.  Recall that this induces a surjection $\wt{\bG{G}}^\ast\rightarrow\bG{G}^\ast$.

We will write $\phi$ for a fixed nontrivial graph automorphism of $G$ normalizing a maximally split torus $T=\bG{T}^F$ and let $F_p$ denote the field automorphism induced by the map $x\mapsto x^p$ on $\overline{\mathbb{F}}_p$, so that $\aut(G)\cong G^\ast\rtimes \langle \phi, F_p\rangle$.  We begin by describing when the group $GQ/Z$ has a self-normalizing Sylow $2$-subgroup, for some $2$-group $Q\in\aut(G)$.  First, we describe the situation when $Q=1$, which can be found in \cite[Theorem 6]{kondratievmazurov}.

\begin{lemma}[Kondrat'ev-Mazurov]\label{lem:whenSN2Ssimple}
A simple group $E_6^\epsilon(q)$ has a self-normalizing Sylow $2$-subgroup if and only if $(q-\epsilon)_{2'} = (3,q-\epsilon)_{2'}$.
\end{lemma}

In fact, from \cite[Lemma 1.3 and proof of Theorem 6]{kondratievmazurov} and \cite[Lemma 4.13]{navarro-tiep:2015:irreducible-representations-of-odd-degree}, we see that for $P$ a Sylow $2$-subgroup of $G$, we have \[N_{G/Z}(PZ/Z)=PZ/Z \times \overline{C}\quad\hbox{ and }\quad N_{G}(P)=P \times {C},\] where $C\cong C_{(q-\epsilon)_{2'}}$ and $\overline{C}\cong C_{(q-\epsilon)_{2'}/(3,q-\epsilon)}$ is a subgroup of $C_{G^\ast}(t)$, with $t$ the unique involution in the center of $PZ/Z$.  By \cite[Proofs of Lemmas 4.25, 4.26]{harris}, we see that in fact $\overline{C}$ is comprised of elements $h(\chi)$ such that $\chi(\alpha_i)=1$ for $2\leq i\leq 5$ and $\chi(\alpha_6)=\chi(\alpha_1)^{-1}$, where we write $\{\alpha_1,...,\alpha_6\}$ for the simple roots numbered as in \cite[13.3.3]{Carter1} and $h(\chi)$ is as in \cite[Section 7.1]{Carter1}.  This, together with the definition of the action of $\phi$ and $F_p$ on the Chevalley generators yields the following:

\begin{lemma}\label{lem:autactC}
The graph automorphism $\phi$ acts on $\overline{C}$ and $C$ by inversion.  Further, the field automorphism $F_p$ acts on $\overline{C}$ and $C$ by $z\mapsto z^p$.
\end{lemma}

\begin{lemma}\label{lem:whenSN2S}
Write $p^a$ for $q$ or $q^2$ in the case $\epsilon =1$ or $\epsilon =-1$, respectively, and let $Q \leqslant \Aut(G)$ be a $2$-group. The quotient $GQ/Z$ has a self-normalizing Sylow $2$-subgroup if and only if at least one of the following is satisfied:
\begin{enumerate}
	\item $G/Z$ has a self-normalizing Sylow $2$-subgroup; 
	\item  $Q$ contains a graph automorphism in case $\epsilon = 1$ or an involutary field automorphism in case $\epsilon=-1$, either of which we may identify as the map $\phi$, up to inner and diagonal automorphisms;
	\item $\epsilon=1$, $Q$ contains a field or graph-field automorphism $\varphi$ of order $a/m$ for some $m$ dividing $a$ (which we identify with $\phi^{\delta'} F_{ p}^m$, up to inner and diagonal automorphisms for $\delta'\in\{0,1\}$), and $\bG{G}^\varphi/Z(\bG{G}^\varphi)$ has a self-normalizing Sylow $2$-subgroup.   That is, by \prettyref{lem:whenSN2Ssimple}, one of the following holds, where we write $\delta=(-1)^{\delta'}$:
	\begin{itemize}
	\item[(3a)] $a$ is a $2$-power and  $(p^m-\delta)_{2'}=1$, or
	\item[(3b)]  $(p^m-\delta)_{2'}=\gcd(3,p^m-\delta)_{2'}$;
	\end{itemize}
		\item  $\epsilon=1$, $Q$ contains a field automorphism $\varphi_1$ of order $a/m_1$ and a graph-field automorphism $\varphi_2$ of order $a/m_2$ for some $m_1,m_2$ dividing $a$ (which we identify with $F_p^{m_1}$ and $\phi F_{ p}^{m_2}$, up to inner and diagonal automorphisms), and $\gcd(p^{m_1}-1, p^{m_2}+1)_{2'}=1$.
	\end{enumerate}
\end{lemma}

We remark that \prettyref{lem:whenSN2S} is quite similar to the corresponding statement for type $A^\epsilon_n(q)$ in \cite{SFTaylorTypeA}.

\begin{proof}[Proof of \prettyref{lem:whenSN2S}]
Armed with \prettyref{lem:autactC}, the proof is analogous to a simplified version of that of \cite[Lemma 8.5]{SFTaylorTypeA}, but we include it for completeness.  Let $P$ be a Sylow $2$-subgroup of $G$ stabilized by $Q$.   

First suppose that one of (1), (2), (3), or (4) holds.  Note that in case (1), the statement is certainly true, since then $N_G(P)=PZ$, so $C_{N_G(P)/PZ}(Q)=1$.  Hence we may assume that $G/Z$ does not have a self-normalizing Sylow $2$-subgroup and that $Q$ contains an outer automorphism.  Specifically, either $Q$ contains a graph automorphism (in case $\epsilon =1$) or involutary field automorphism (in case $\epsilon = -1$), which we identify with $\phi$ on $G$, up to conjugation in $\widetilde{G}$; or $\epsilon=1$ and $Q$ contains a field or graph-field automorphism, which we identify as $\phi^{\delta'}F_{p^m}$ on $G$, up to conjugation in $\widetilde{G}$, for some $m\geq 1$.  Write $\varphi$ for the corresponding graph, field, or graph-field automorphism, respectively.  We will show that $C_{N_G(P)/PZ}(\varphi)=1$ in cases (2) and (3) and that $C_{N_G(P)/PZ}(\varphi_1, \varphi_2)=1$ in case (4).  Write $\overline{N}:=N_G(P)/PZ$ and let $\overline{g}$ denote the image of an element $g\in N_G(P)$ in $\overline{N}$.  Suppose $g\in N_G(P)$ satisfies that $\overline{g}\in C_{\overline{N}}(\varphi)$.  That is, $\overline{g}$ is fixed by $\varphi$ (resp. $\varphi_1$ and $\varphi_2$).  

Now, we have $g=xz$ for some $x\in P$ and $z\in C$.  Then observing the action of $\varphi$ on the $2'$-part of $g$, we see $\varphi(z)=zy$ for some $y\in Z$ of odd order in case (2) or (3).  Hence in cases (2) or (3a), there is some integer $c\geq 1$ such that $z^{2^c}=y$.  Since $z$ has odd order, we therefore see that $g\in PZ$, so that $\overline{g}=1$, yielding that in cases (2) and (3a), $C_{\overline{N}}(\varphi)=1$.  Similarly, in case (4), for $i=1,2$, we have $\varphi_i(z)=zy_i$ for some $y_i\in Z$ of odd order, so since $\gcd(p^{m_1}-1, p^{m_2}+1)_{2'}=1$, we also see in case (4) that $g\in PZ$ and $C_{\overline{N}}(\varphi_1, \varphi_2)=1$.

 Now assume condition (3b) holds, so that $\epsilon=1$, $(p^m-\delta)_{2'}=\gcd(3, p^m-\delta)_{2'}$, and $E_6^\delta(p^m)$ and $E_6^\delta(p^m)_{ad}$ have self-normalizing Sylow $2$-subgroups.  In this case, $z^{p^m-\delta}$ is an element of $Z$ of odd order.  In particular, the image of $z$ in $G/Z$ may be viewed as a member of $E_6^\delta(p^m)$ which centralizes a Sylow $2$-subgroup $P_m$ of $E_6^\delta(p^m)$ contained in $PZ/Z$.  But since $E_6^\delta(p^m)$ has a self-normalizing Sylow $2$-subgroup, we see that the image of $z$ in $G/Z$ is trivial, so that $\overline{g}=1$, and $C_{\overline{N}}(\varphi)=1$ in case (3b) as well.

Now, assume that none of (1) to (4) hold.  Then $(q-\epsilon)_{2'}>(3, q-\epsilon)_{2'}$ by \prettyref{lem:whenSN2Ssimple} and $\phi$ is not contained in $Q$ up to conjugation in $\wt{G}$.  Further, since the diagonal automorphisms of $S$ are odd, it follows that $Q$ does not contain a diagonal automorphism.
 
 It suffices to exhibit a $z\in C$ such that $\varphi(z)=z$ for each field or graph-field automorphism $\varphi$ contained in $Q$.   If $Q$ does not contain a field or graph-field automorphism, then we may choose such a $z$ to be a generator of $\overline{C}$.  In particular, this is the case if $\epsilon=-1$, so we assume $\epsilon=1$.  Write $q=p^a$.  If $GQ/G$ is cyclic containing a field or graph-field automorphism, we may identify the generator of the subgroup of $Q$ consisting of field or graph-field automorphisms as $\phi^{\delta'}F_p^m$ for some $m|a$ and $\delta'\in\{0,1\}$.  Similarly, if $Q$ contains both a field automorphism and graph-field automorphism, we may assume the corresponding subgroups are generated by $\varphi_1:=F_p^{m_1}$ and $\varphi_2:=\phi F_p^{m_2}$ for some $m_1, m_2|a$.  Without loss, it suffices to assume that these generators generate the largest $2$-group of automorphisms possible without inducing conditions (3) and (4).  In the cyclic case, we may choose $z$ to be an element of $\overline{C}\cap E^\delta_6(p^m)$.  Otherwise, we may take $z$ to be a generator of the subgroup of $C$ of size $\gcd(p^{m_1}-1, p^{m_2}+1)_{2'}$, which can be viewed as a member of $\overline{C}$ since (3) does not hold. Hence $z$ is fixed by $Q$ and has the required form. 
 \end{proof}
\begin{theorem}\label{thm:typeE6}
The simple groups $E_6^\epsilon(q)$ are SN2S-Good for odd $q$.
\end{theorem}
\begin{proof}

We begin by showing \prettyref{cond:conjFI} holds for the simple groups $E_6^\epsilon(q)$ for $q$ odd.  Let $A=SQ$ be an almost simple group obtained by adjoining a $2$-group $Q$ of automorphisms to $S=E_6^\epsilon(q)$.  Suppose that $A$ does not contain a self-normalizing Sylow $2$-subgroup, so that $Q$ is not as in (1)-(4) of \prettyref{lem:whenSN2S}.  

By \cite[Lemma 4.13]{navarro-tiep:2015:irreducible-representations-of-odd-degree}, we know that every odd-degree character of $G^\ast$ not lying over a unipotent character restricts irreducibly to $S$, and that every non-unipotent odd-degree character of $S$ can be obtained in this way.  Hence it suffices to exhibit a $Q$-invariant $\wt{\chi}\in\irr_{2'}(G^\ast)$ which is not fixed by $\sigma$.      Letting $s$ be the pre-image in $C$ of the element $z$ of $\overline{C}$ obtained in the last paragraph of the proof of \prettyref{lem:whenSN2S}, we see that the semisimple character $\chi_s$ in the rational Lusztig series $\mathcal{E}(G^\ast, s)$ has odd degree, is $Q$-invariant, and is not fixed by $\sigma$ (see \cite[Corollary 2.4]{NavarroTiepTurullCyclo} and \cite[Lemma 3.4]{SFTaylorTypeA}).

Next we prove that \prettyref{cond:conjIF} holds for the simple groups $E_6^\epsilon(q)$ for $q$ odd.  Let $Q$ be as in \prettyref{lem:whenSN2S}, so that $GQ/Z$ has a self-normalizing Sylow $2$-subgroup.   Let $\chi\in\irr_{2'}(G)$ be non-unipotent and fixed by $Q$, and suppose that the rational Lusztig series $\mathcal{E}(G,s)$ contains $\chi$.  By \cite[Lemma 4.13]{navarro-tiep:2015:irreducible-representations-of-odd-degree}, $\chi$ extends to an irreducible character of $\wt{G}$.  We claim that it suffices to show that the series $\mathcal{E}(G,s)$ is stabilized by $\sigma$.  Indeed, in this case, note that $\wt{\chi}^\sigma\in\irr(\wt{G}|\chi^\sigma)$ if  $\wt{\chi}\in\irr(\wt{G}|\chi)$.  Note that $\wt{\chi}\in\mathcal{E}(\wt{G},\wt{s})$ for some pre-image $\wt{s}\in\wt{G}^\ast$ of $s$ by \cite[Corollaire 9.5]{bonnafe06}. Similarly, $\wt{\chi}^\sigma\in\mathcal{E}(\wt{G},\wt{s}t)$ for some $t\in Z(\wt{G}^\ast)$.  So by \cite[Proposition 13.30]{dignemichel}, $\wt{\chi}^\sigma=\wt{\chi}' \otimes\wh{t}$ for some linear character $\wh{t}$ of $\wt{G}^\ast$ and $\wt{\chi}'\in\mathcal{E}(\wt{G},\wt{s})$.  But the proof of \cite[Lemma 4.13]{navarro-tiep:2015:irreducible-representations-of-odd-degree} shows that the character degrees amongst members of $\mathcal{E}(\wt{G},\wt{s})$ lying above odd-degree characters of $G$ are distinct, so that $\wt{\chi}=\wt{\chi}'$ and $\chi^\sigma=\wt{\chi}^\sigma|_G=\wt{\chi}|_G=\chi$, proving the claim.

In particular, note that $\mathcal{E}(G,s)$ is stabilized by $\sigma$ in the case that $s$ has $2$-power order, by \cite[Lemma 3.4]{SFTaylorTypeA}.  If case (1) of \prettyref{lem:whenSN2S} holds, then we are done since $C_{G^\ast}(s)$ must contain a Sylow $2$-subgroup of $G^\ast$, which is self-normalizing and hence self-centralizing.  In the remaining cases, we may write $R_T^G(\lambda)$ for the Harish-Chandra series containing $\chi$, where $T=\bG{T}^F$ is a maximally split torus, by \cite[Lemma 7.9]{MalleSpathMcKay2}.   By \prettyref{lem:lambdafixseriesfix}, it suffices to show that $\lambda^\sigma=\lambda$.  However, since $\chi$ is fixed by $Q$, it is clear that $R_T^G(\lambda)$, and hence $\lambda$, is fixed by $Q$, since $\langle \phi, F_p\rangle$ stabilizes $T$.  But notice that $\phi$ acts on $T$ via inversion and $F_p$ acts on $T$ via $t\mapsto t^{p}$, so that in cases (2), (3a), or (4), we see that $\lambda$ must have $2$-power order, so is fixed by $\sigma$.  In case (3b), we may view $\lambda$ as a character of the maximally split torus $\bG{T}^{\varphi}$ of $\bG{G}^\varphi$, so that $\lambda\in\mathcal{E}(\bG{T}^\varphi, s)$ and $s$ has $2$-power order since $s\in (\bG{T}^\varphi)^\ast\leq (\bG{G}^\varphi)^\ast$, which has a self-normalizing Sylow $2$-subgroup.  
\end{proof}

\section{Acknowledgements}

The author would first and foremost like to express her gratitude to Britta Sp{\"a}th, who suggested the main strategy of analyzing the action on Harish-Chandra series, provided a number of extremely useful conversations, and without whose help the original manuscript likely would not have existed.  She is also extremely grateful to the anonymous referee, whose careful reading lead to detailed comments and suggestions that vastly improved and clarified the exposition, particularly in Section 2. 

The author was supported in part by a grant through MSU Denver's Faculty Scholars Program, a grant from the Simons Foundation (Award \#351233), and an NSF-AWM Mentoring Travel Grant. She would like to thank G.\ Malle and B. Sp{\"a}th for their hospitality and helpful discussions during the research visit to TU Kaiserslautern supported by the latter grant, during which the majority of the work for this article was accomplished.  She would also like to thank G. Malle, B. Sp{\"a}th, and J. Taylor for their extremely helpful comments on an earlier version of the manuscript. 

Part of this was completed while the author was in residence at the Mathematical
Sciences Research Institute in Berkeley, California during the Spring 2018 semester program
on Group Representation Theory and Applications, supported by the National Science Foundation
under Grant No. DMS-1440140. She thanks the institute and the organizers of the program for
making her stay possible and providing a collaborative and productive work environment.

In addition, the author would like to thank P. H. Tiep and G. Navarro for suggesting this problem and their kindness and many helpful conversations throughout the three articles detailing her work on \prettyref{conj:mainprob}.

\bibliographystyle{alpha}
\bibliography{researchreferences}

\end{document}